\documentclass[letterpaper, 11pt]{amsart}
\usepackage{mathtools}
\usepackage{amsmath}
\usepackage[T1]{fontenc}
\usepackage{amssymb}
\usepackage{yhmath}
\usepackage{comment}
\usepackage{graphicx} 
\usepackage{ mathrsfs }
\usepackage{bbm}
\usepackage{xcolor}
\usepackage{tikz-cd}
\usepackage{tikz}
\usetikzlibrary{patterns}
\usepackage{hyperref}

\DeclareMathAlphabet{\mathpzc}{OT1}{pzc}{m}{it}

\usepackage{thmtools}
\usepackage{thm-restate}

\usepackage{caption}

\newtheorem{theorem}{Theorem}[section]

\newtheorem*{claim*}{Claim}

\newtheorem{lemma}[theorem]{Lemma}
\newtheorem{lem}[theorem]{Lemma}
\newtheorem{corollary}[theorem]{Corollary}

\newtheorem{cor}[theorem]{Corollary}
\newtheorem{question}[theorem]{Question}

\newtheorem{proposition}[theorem]{Proposition}

\newtheorem{prop}[theorem]{Proposition}

\newtheorem{fact}[theorem]{Fact}

\theoremstyle{definition}
\newtheorem{definition}[theorem]{Definition}

\newtheorem{example}[theorem]{Example}

\theoremstyle{remark}
\newtheorem{remark}[theorem]{Remark}

\numberwithin{equation}{section}

\newcommand{\abs}[1]{\lvert#1\rvert}
\newcommand{\norm}[1]{\lVert#1\rVert}
\newcommand{\op}{\operatorname}

\newcommand{\be}{\begin{equation}}
\newcommand{\ee}{\end{equation}}
\newcommand{\Ga}{\Gamma}

\newcommand{\R}{\mathbb R}
\renewcommand{\H}{\mathbb H}
\newcommand{\Z}{\mathbb Z}
\newcommand{\N}{\mathbb N}
\newcommand{\ga}{\gamma}

\newcommand{\la}{\lambda}
\newcommand{\La}{\Lambda}
\newcommand{\inte}{\op{int}}
\newcommand{\ba}{\backslash}
\newcommand{\ov}{\overline}

\newcommand{\cal}{\mathcal}

\newcommand{\SO}{\op{SO}}
\newcommand{\Isom}{\op{Isom}}
\newcommand{\PSL}{\op{PSL}}

\newcommand{\stab}{\op{Stab}}

\newcommand{\diam}{\op{Diam}}

\newcommand{\BR}{\op{BR}}

\renewcommand{\L}{\mathcal L}

\renewcommand{\S}{\mathbb S}

\newcommand{\so}{\SO^\circ}

\newcommand{\id}{\op{id}}

\newcommand{\E}{\cal E}

\renewcommand{\P}{\mathbb{P}}
\newcommand{\ess}{\mathsf{E}}

\newcommand{\rank}{\op{rank}}
\newcommand{\Lie}{\op{Lie}}

\newcommand{\Hor}{\mathcal{H}}

\newcommand{\Lip}{\operatorname{Lip}}

\newcommand{\CAT}{\operatorname{CAT}}

\newcommand{\Spec}{\operatorname{Spec}}

\newcommand{\fa}{\mathfrak{a}}

\begin{document}

\title[Horospherical measures in higher rank: Teaser]{Classification of horospherical invariant measures in higher rank: Teaser}

\author{Inhyeok Choi}
\address{School of Mathematics, KIAS, Hoegi-ro 85, Dongdaemun-gu, Seoul 02455, South Korea}
\email{inhyeokchoi48@gmail.com}

\author{Dongryul M. Kim}
\address{Department of Mathematics, Yale University, New Haven, CT 06511}
\email{dongryul.kim@yale.edu}

\begin{abstract}
    Let $G$ be a product of rank-one simple real algebraic groups and let $\Gamma < G$ be a Zariski dense Anosov subgroup, or relatively Anosov subgroup. In this paper, we prove a complete classification of invariant Radon measures for the maximal horospherical action on $\Gamma \ba G$. In particular, when $\Gamma$ is Anosov, this solves the open problems proposed by Landesberg--Lee--Lindenstrauss--Oh for $\operatorname{rank} G \le 3$, and by Oh in general.

    More generally, we consider the horospherical foliation of a product of $\operatorname{CAT}(-1)$ spaces, and present a classification of Radon measures supported on a recurrent subfoliation that are invariant under the action of transverse subgroups.

    \bigskip

    \bigskip

    \bigskip

\begin{center}
{\bf {\large D}ECLARATION}
\end{center}

\bigskip

This paper is a teaser of the paper \cite{CK_endgame}, which is for arbitrary semisimple real algebraic groups. We decided to keep the current paper public on arXiv and our websites, as many parts of the argument in \cite{CK_endgame} are simplified, which might make this paper more accessible to readers who are not familiar with homogeneous dynamics but have some geometry background.

Importantly, we also decided {\bf not to  publish this paper in any  journal,} having \cite{CK_endgame} as our main mathematical result.

\end{abstract}

\maketitle
\tableofcontents

\section{Introduction}

Given a dynamical system, classifying invariant measures is a natural and important questions with many applications, as also indicated by the celebrated theorem of Ratner \cite{Ratner_measure}. We study this question for dynamical systems given by \emph{horospherical actions}.

Let $G$ be a connected semisimple real algebraic group and $P < G$ its minimal parabolic subgroup with a fixed Langlands decomposition $P = MAN$, where $A$ is a maximal real split torus of $G$, $M < P$ is a maximal compact subgroup commuting with $A$, and $N$ is the unipotent radical of $P$.

Let $\Ga < G$ be a Zariski dense discrete subgroup. The right multiplication of $N$ on $\Ga \ba G$ is called (maximal) \emph{horospherical action}. 
For a uniform lattice $\Ga < G$, the $N$-action on $\Ga \ba G$ is uniquely ergodic,\footnote{By unique ergodicity, we mean that there exists a unique invariant ergodic Radon measure up to a constant multiple.} with the Haar measure for $G$ as the ergodic measure. This was first shown for $G = \PSL(2, \R)$ by Furstenberg \cite{furstenberg1973the-unique}, and by Veech \cite{veech1977unique} in general. When $\Ga < G$ is a non-uniform lattice, Dani classified all $N$-invariant ergodic Radon measures on $\Ga \ba G$ (\cite{dani1978invariant}, \cite{dani1981invariant}).

We are mainly interested in the case that $\Ga < G$ is not a lattice, i.e., $\Ga$ has infinite covolume. The first classification of horospherical invariant measure in this setting is due to Burger \cite{Burger_horoc}, who considered the case that $G = \PSL(2, \R)$ and $\Ga < G$ is convex cocompact with critical exponent strictly bigger than $1/2$. 
More generally, when $G$ is of rank one and $\Ga < G$ is geometrically finite, Roblin classified all $NM$-invariant ergodic Radon measures on $\Ga \ba G$ \cite{Roblin2003ergodicite}. The main component of the works of Burger and Roblin is that the $NM$-action is uniquely ergodic on the \emph{recurrence locus}, the subset of $\Ga \ba G$ where the forward frame flow (or geodesic flow) is recurrent to a compact subset.
This unique ergodic measure is now called the \emph{Burger--Roblin measure}. Later, Winter showed that the Burger--Roblin measure is $N$-ergodic and provided the classification of $N$-invariant Radon measures \cite{Winter_mixing}. 
For geometrically infinite cases, Babillot and Ledrappier first discovered that there may be continuous family of $NM$-invariant ergodic Radon measures (\cite{Babillot_nilpotent}, \cite{BL_covers}); see also (\cite{Sarig_abelian}, \cite{Sarig_genus}, \cite{Ledrappier_invariant}, \cite{LS_periodic}, \cite{Winter_mixing}, \cite{OP_local}, \cite{LL_Radon}, \cite{Landesberg_horospherically}, \cite{LLLO_Horospherical})  for partial classification
results in the rank-one case.

We now move to the case that $G$ is of \emph{higher rank}. Edwards--Lee--Oh extended the notion of Burger--Roblin measure to higher rank, introducing higher-rank Burger--Roblin measures \cite{Edwards2020anosov}. Their ergodicity with respect to horospherical actions were proved for Zariski dense Borel Anosov subgroups by Lee--Oh (\cite{LO_invariant}, \cite{LO_ergodic}), and for a larger class of discrete subgroups by the second author \cite{kim2024conformal}.
Later in this paper, we will also generalize this ergodicity to horospherical foliations of products of $\CAT(-1)$ spaces using a different approach (Theorem \ref{thm:ergodicdiv}).

On the other hand, the only known result towards measure classification in higher-rank settings was the work of Landesberg--Lee--Lindenstrauss--Oh \cite{LLLO_Horospherical}. They considered 
\begin{equation} \label{eqn:standing}
G := \prod_{i = 1}^r G_i
\end{equation}
where $G_i$ is a simple real algebraic group of rank one. In this case, we have $r = \rank G$. We also assume Equation \eqref{eqn:standing} in the rest of the introduction.

They also considered the directionally recurrent set in $ \Ga \ba G$ for each \emph{1-dimensional diagonal flow} (or, directional flow). More precisely, denote by $\fa := \Lie A$ and fix a positive Weyl chamber $\fa^+ \subset \fa$. Then for each $v \in \inte \fa^+$, they showed that up to scaling, there exists at most one $N$-ergodic invariant Radon measure supported on $\mathcal{R}_{\Ga, v} \subset \Ga \ba G$ consisting of elements each of whose 1-dimensional $\exp (\R_{>0} v)$-orbit is recurrent to a compact subset. 

On the other hand, whether $\mathcal{R}_{\Ga, v}$ supports a nonzero, $N$-invariant Radon measure or not is understood only when $\Ga <G $ is an \emph{Anosov subgroup}. An Anosov subgroup is a higher-rank generalization of convex cocompact subgroups, introduced by Labourie \cite{Labourie2006anosov} for surface groups and generalized by Guichard--Wienhard \cite{Guichard2012anosov} for hyperblic groups. 

In the setting of product of rank-one Lie groups, $\Ga <G $ is (Borel)\footnote{i.e., with respect to a minimal parabolic subgroup. Throughout the paper, we only consider this case, and similarly for relatively Anosov and transverse subgroups.} Anosov if 
the projection $\Ga \to G_i$ has finite kernel and convex cocompact image for all $1 \le i \le r$. Based on the ergodicity results of Lee--Oh (\cite{LO_invariant}, \cite{LO_ergodic}) and Burger--Landesberg--Lee--Oh \cite{BLLO}, the rigidity result of \cite{LLLO_Horospherical} is as follows:

\begin{theorem}[{\cite{LLLO_Horospherical}}] \label{thm:LLLO}
    Let $\Ga <G $ be a Zariski dense Anosov subgroup and $v \in \inte \fa^+$. Let $\L_{\Ga} \subset \fa^+$ denote the limit cone\footnote{The limit cone of $\Ga$ is the asymptotic cone of the Cartan projections of $\Ga$ in $\fa$. We  will revisit this later.} of $\Ga$.
    \begin{enumerate}
        \item For $r \le 3$ and $v \in \inte \L_{\Ga}$, the $N$-action on $\mathcal{R}_{\Ga, v}$ is uniquely ergodic.
        \item For $r > 3$ or $v \notin \inte \L_{\Ga}$, there exists no non-zero, $N$-invariant measure Radon measure supported on $\mathcal{R}_{\Ga, v}$.
    \end{enumerate}
\end{theorem}

The ergodic measures in (1) above are higher-rank Burger--Roblin measures, whose ergodicity was proved in \cite{LO_ergodic}, and being supported on the directionally recurrent set was proved in \cite{BLLO}. Delaying their definitions, we  note that in contrast to rank-one settings, they come as a family of mutually singular measures, because higher-rank Patterson--Sullivan measures do so. The reason for the rank dichotomy in Theorem \ref{thm:LLLO}(2) is that $\mathcal{R}_{\Ga, v}$ has zero Burger--Roblin measures when $r > 3$ \cite{BLLO}.

A genuine region for the horospherical action is the unique $P$-minimal set
$$\E_{\Ga} \subset \Ga \ba G$$
where the uniqueness is due to Benoist \cite{Benoist1997proprietes}.
In view of Theorem \ref{thm:LLLO}, the following open problem was proposed by Landesberg--Lee--Lindenstrauss--Oh,  towards classifying horospherical invariant measures.

\begin{question}[{\cite[Open problem 1.8]{LLLO_Horospherical}}] \label{ques:LLLO}
Let $\Ga < G$ be a Zariski dense Anosov subgroup and suppose $r \le 3$. Is any $N$-invariant ergodic Radon measure on $\E_{\Ga}$ supported on $\mathcal{R}_{\Ga, v}$ for some $v \in \inte \fa^+$?
\end{question}

More generally, in a very recent preprint for the Proceedings of the ICM 2026, Oh asked for horospherical measure classification for Anosov subgroups without any rank assumption on $G$, i.e., on $r$.

\begin{question}[{\cite[Section 8.2]{oh2025dynamics}}] \label{ques:Anosov}
Let $\Ga < G$ be a Zariski dense Anosov subgroup. Is any $N$-invariant ergodic Radon measure on $\E_{\Ga}$ a Burger--Roblin measure?
\end{question}

\subsection{Main results for Anosov subgroups}

Main results of this paper are affirmative answers to Question \ref{ques:LLLO} and Question \ref{ques:Anosov}, resolving the open problem proposed by Landesberg--Lee--Lindenstrauss--Oh in \cite{LLLO_Horospherical} and by Oh in \cite{oh2025dynamics}. Indeed, we give a complete classification of horospherical invariant measures.

\begin{theorem} \label{thm:mainAnosov}
    Let $\Ga < G$ be a Zariski dense Anosov subgroup. Let $\mu$ be a non-zero, $N$-invariant ergodic Radon measure on $\Ga \ba G$. Then either
    \begin{enumerate}
        \item $\mu$ is supported on $\E_{\Ga}$ and is a constant multiple of a Burger--Roblin measure, or 
        \item $\mu$ is supported on a closed $NM$-orbit in $(\Ga \ba G) \smallsetminus \E_{\Ga}$. 
    \end{enumerate}
\end{theorem}
Note that the same holds for $NM$-invariant ergodic Radon measures (see Corollary~\ref{cor:mainAnosov} below).

\begin{remark}
    Under an extra assumption that the measure is $AM$-quasi-invariant, the measure classification was proved by Lee--Oh (\cite[Theorem 1.1]{LO_invariant}, \cite[Theorem 1.3]{LO_ergodic}) for Anosov subgroups as above, and by the second author \cite{kim2024conformal} for relatively Anosov subgroups and transverse subgroups as in Theorem \ref{thm:maintransverse} below. These are consequences of $N$-ergodicity of Burger--Roblin measures.

    The major part of the proof our main results is to show that any $NM$-invariant ergodic Radon measure on $\E_{\Ga}$ is $A$-quasi-invariant. Once we have the quasi-invariance, then the classification follows from \cite[Proposition 10.25]{LO_invariant}.
    See also \cite[0.1 Basic Lemma]{aaronson2002invariant} and \cite[Lemma 1]{Sarig_abelian} for this in a more abstract setting.
\end{remark}

In fact, we classify horospherical invariant measures for a more general class of discrete subgroups. Delaying this general result to the next subsection, we first describe higher-rank Burger--Roblin measures.

Fix a maximal compact subgroup $K < G$ so that the Cartan decomposition $G = K (\exp \fa^+) K$ holds. Then we have the Furstenberg boundary
$$
\mathcal{F} := K/M = G/P.
$$
Let $\Ga < G$ be a Zariski dense discrete subgroup.
For $\delta \ge 0$ and a linear form $\psi \in \fa^*$, a Borel probability measure $\nu$ on $\mathcal{F}$ is called a $\delta$-dimensional $\psi$-\emph{conformal measure} of $\Ga$ if
$$
\frac{d g_* \nu}{d \nu}(\xi) = e^{-\delta \cdot \psi(\beta_{\xi}(g, \id))} \quad \text{a.e.}
$$
where $\beta$ is the $\fa$-valued Busemann cocycle (Equation \eqref{eqn:Busedef}), each of whose components is a usual Busemann cocycle for a rank-one symmetric space. This notion of conformal measures was first introduced by Quint \cite{Quint2002Mesures}, generalizing the classical Patterson--Sullivan theory to higher rank.

In \cite{Edwards2020anosov}, Edwards--Lee--Oh extended the classical Burger--Roblin measure to higher rank. For a $\delta$-dimensional $\psi$-conformal measure $\nu$ of $\Ga$ on $\mathcal{F}$, the (higher-rank) \emph{Burger--Roblin measure} associated to $\nu$ is the Radon measure $\mu_{\nu}^{\BR}$ on $\Ga \ba G$ induced by the $\Ga$-invariant measure $\tilde \mu_{\nu}^{\BR}$ on $G$ defined as follows: for $g = k (\exp u) n \in K(\exp \fa) N$ in Iwasawa decomposition of $G$,
\begin{equation} \label{eqn:BRdef}
d \tilde \mu_{\nu}^{\BR}(g) := e^{\delta \cdot \psi(u)} d \tilde \nu (k) du dn
\end{equation}
where $\tilde \nu$ is the $M$-invariant lift of $\nu$ to $K$ and $du$ and $dn$ are Lebesgue measures on $\fa$  and $N$ respectively. The measure $\mu_{\nu}^{\BR}$ is $NM$-invariant.

We denote by $\La(\Ga) \subset \mathcal{F} = G/P$ the limit set of $\Ga$, which is the unique $\Ga$-minimal subset \cite{Benoist1997proprietes}. In terms of the limit set, we have
$$
\E_{\Ga} = \{ [g] \in \Ga \ba G : g P \in \La(\Ga) \}.
$$
Hence, $\mu_{\nu}^{\BR}$ is supported on $\E_{\Ga}$ if and only if $\nu$ is supported on $\La(\Ga)$, and in this case, the $NM$-ergodicity and $N$-ergodicity were proved by Lee--Oh (\cite{LO_invariant}, \cite{LO_ergodic}).
As a corollary of Theorem \ref{thm:mainAnosov} we conclude that Burger--Roblin measures are all such ergodic measures.

\begin{corollary} \label{cor:mainAnosov}
    Let $\Ga < G$ be a Zariski dense Anosov subgroup.  Then the following three sets are the same, up to constant multiples:
    \begin{enumerate}
        \item $\{ \mu_{\nu}^{\BR} : \nu \text{ is a conformal measure of $\Ga$ on $\La(\Ga)$}\}$.
        \item the set of all $NM$-invariant ergodic Radon measures on $\E_{\Ga}$.
        \item the set of all $N$-invariant ergodic Radon measures on $\E_{\Ga}$.
    \end{enumerate}
\end{corollary}

The set of $N$-ergodic measures in Corollary \ref{cor:mainAnosov} can be described more explicitly. 
Denote by $\kappa : G \to \fa^+$ the Cartan projection, defined by the condition $g \in K (\exp \kappa(g)) K$ for all $g \in G$.
The \emph{limit cone} $\L_{\Ga} \subset \fa^+$ of $\Ga$ is the asymptotic cone of Cartan projections $\kappa(\Ga)$. Benoist showed that if $\Ga$ is Zariski dense, $\L_{\Ga}$ is convex and has non-empty interior \cite{Benoist1997proprietes}. For a Zariski dense Anosov subgroup $\Ga < G$, Lee--Oh classified conformal measures of $\Ga$ on $\La(\Ga)$ in \cite{LO_invariant}, and provided a natural homeomorphism
\begin{equation} \label{eqn:LObijection}
\P(\inte \L_{\Ga}) \quad \longleftrightarrow \quad \{ \mu_{\nu}^{\BR} : \nu \text{ is a conformal measure of $\Ga$ on $\La(\Ga)$}\}
\end{equation}
constructed using tangencies of the growth indicator of $\Ga$,  introduced by Quint \cite{Quint2002divergence}.
Corollary \ref{cor:mainAnosov} is now rephrased as follows:

\begin{corollary}  \label{cor:mainAnosovrephrase}
    Let $\Ga < G$ be a Zariski dense Anosov subgroup. Then the homeomorphism in Equation \eqref{eqn:LObijection} becomes homeomorphisms among the following three sets:
    $$
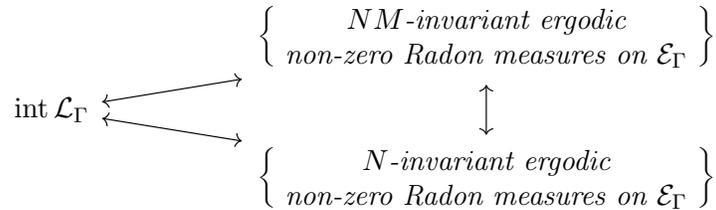
\begin{tikzcd}[column sep = huge, row sep = 0.1em]
    & \left\{
        \ \begin{matrix}
            NM\text{-invariant ergodic}\\
            \text{non-zero Radon measures on } \E_{\Ga}
        \end{matrix} \ \right\} \arrow[dd, <->] \\
    \inte \L_{\Ga} \arrow[ur, <->] \arrow[dr, <->] & \\
    & \left\{
        \ \begin{matrix}
            N\text{-invariant ergodic}\\
            \text{non-zero  Radon measures on } \E_{\Ga}
        \end{matrix} \ \right\}
    \end{tikzcd}
    $$
    In particular, they are all homeomorphic to $\R^{r}$.
\end{corollary}

Corollary \ref{cor:mainAnosov} and Corollary \ref{cor:mainAnosovrephrase} do not have any rank assumption, and hence gives an affirmative answer to Question \ref{ques:Anosov}, which was proposed by Oh in \cite{oh2025dynamics}. For Question \ref{ques:LLLO}, we note that $\mathcal{R}_{\Ga, v} \subset \E_{\Ga}$ is the same as the set of $[g] \in \Ga \ba G$ such that $gP \in \mathcal{F}$ is contained in the ``directional limit set for $v$,'' which is a subset of $\La(\Ga)$. In \cite{BLLO}, it was shown that when $r \le 3$, any conformal measure $\nu$ of a Zariski dense Anosov subgroup $\Ga < G$ supported on $\La(\Ga)$ is in fact supported on the directional limit set for some $v \in \inte \fa^+$. Therefore, Question \ref{ques:LLLO}, the open problem proposed in \cite{LLLO_Horospherical}, is resolved by Corollary \ref{cor:mainAnosov}.

\subsection{Beyond Anosov subgroups}

Our approach to measure classification applies to subgroups beyond Anosov ones, namely the transverse subgroups.

The notion of transverse subgroups of general Lie groups was introduced and studied by Canary--Zhang--Zimmer \cite{CZZ_transverse}. This notion extends rank-one discrete subgroups to higher rank, and Anosov subgroups are special examples of transverse subgroups. Their important property that they admit natural convergence group actions was proved earlier by Kapovich--Leeb--Porti \cite{KLP_Anosov}.

We now define transverse subgroups in our setting of Equation \eqref{eqn:standing}. Note that the associated Riemannian symmetric space $G/K$ and the Furstenberg boundary $\mathcal{F}$ can be written as 
$$
G/K = \prod_{i = 1}^r X_i \quad \text{and} \quad \mathcal{F} = \prod_{i = 1}^r  \partial X_i
$$
where $X_i$ is the rank-one symmetric space associated to $G_i$ and $\partial X_ii$ is its Gromov boundary, for each $1 \le i \le r$. Fix a basepoint $o = [\id] \in G/K$.

\begin{definition} \label{def:transverseintro}
A Zariski dense discrete subgroup $\Ga < G$ is called \emph{transverse} if
\begin{itemize}
    \item for any infinite sequence $\{g_n \}_{n \in \N} \subset \Ga$, we have that $g_n o \in G/K$ diverges as $n \to + \infty$ in each component $X_i$, $1 \le i \le r$, and 
    \item for any two distinct $(\xi_1, \dots, \xi_r), (\zeta_1, \dots, \zeta_r) \in \La(\Ga)$, we have $\xi_i \neq \zeta_i$ for all $1 \le i \le r$.
\end{itemize}
\end{definition}

A Zariski dense transverse subgroup $\Ga < G$ acts on the limit set $\La(\Ga) \subset \mathcal{F}$ as a convergence group. When the $\Ga$-action on $\La(\Ga)$ is a geometrically finite convergence action, we call $\Ga$ \emph{relatively Anosov}. If the $\Ga$-action on $\La(\Ga)$ is a uniform convergence action, then $\Ga$ is \emph{Anosov}, and vice versa.
As Anosov subgroups are higher-rank version of convex cocompact subgroups in rank one, relatively Anosov subgroups are higher-rank analogues of rank-one geometrically finite subgroups.

For this general class of discrete subgroups, we consider a subset $\mathcal{R}_{\Ga} \subset \E_{\Ga}$ which we call \emph{recurrence locus}, defined as follows:
\begin{equation} \label{eqn:recurrencelocus}
\mathcal{R}_{\Ga} := \{ x \in \Ga \ba G : x \cdot \exp \fa^+ \text{ is recurrent to a compact subset} \},
\end{equation}
i.e., $x \in \mathcal{R}_{\Ga}$ if and only if there exists a sequence $\{a_n\}_{n \in \N} \subset \fa^+$ diverging in each component of $\fa =\R^r$ such that $\{x a_n\}_{n \in \N}$ is contained in a fixed compact subset. The set $\mathcal{R}_{\Ga}$ is much larger than $\mathcal{R}_{\Ga, v}$, $v \in \inte \fa^+$, discussed before, because $\mathcal{R}_{\Ga}$ considers the full $\exp \fa^+$-orbits, not only a fixed 1-dimensional one given by $v$. Indeed, when $\Ga$ is Anosov, $\mathcal{R}_{\Ga} = \E_{\Ga} $ while $\mathcal{R}_{\Ga, v}$ is a proper subset of $\E_{\Ga}$.

For a linear form $\psi \in \fa^*$, we denote the associated Poincar\'e series by $\mathcal{P}_{\Ga, \psi}(s) := \sum_{g \in \Ga} e^{-s \psi(\kappa(g))}$. We also denote  its critical exponent by $\delta_{\psi}(\Ga) := \inf \{ s > 0 : \mathcal{P}_{\Ga, \psi}(s) < + \infty \}$.
We say that a conformal measure $\nu$ of $\Ga$ is of \emph{divergence type} if $\nu$ is a $\delta_{\psi}(\Ga)$-dimensional $\psi$-conformal measure of $\Ga$ for some $\psi \in \fa^*$ such that $\delta_{\psi}(\Ga) < + \infty$ and $\mathcal{P}_{\Ga, \psi}(\delta_{\psi}(\Ga)) = + \infty$. 
Note that in our setting, the Cartan projection $\kappa(g) \in \fa = \R^r$ is the vector whose $i$-th component is the displacement between $o, go \in G/K$ in the component $X_i$.

By the higher-rank Hopf--Tsuji--Sullivan dichotomy for transverse subgroups (\cite{CZZ_transverse}, \cite{KOW_PD}), the Burger--Roblin measure $\mu_{\nu}^{\BR}$ is supported on $\mathcal{R}_{\Ga}$ for a divergence-type conformal measure $\nu$. Moreover, in this case, the second author showed $NM$-ergodicity and  $N$-ergodicity of $\mu_{\nu}^{\BR}$ in \cite{kim2024conformal}. It turns out that they are the only ergodic measures.

\begin{theorem} \label{thm:maintransverse}
    Let $\Ga < G$ be a Zariski dense transverse subgroup. Then the following three sets are the same, up to constant multiples:
    \begin{enumerate}
        \item $\{ \mu_{\nu}^{\BR} : \nu \text{ is a divergence-type conformal measure of $\Ga$ on $\La(\Ga)$}\}$.
        \item the set of all $NM$-invariant ergodic Radon measures on $\mathcal{R}_{\Ga}$.
        \item the set of all $N$-invariant ergodic Radon measures on $\mathcal{R}_{\Ga}$.
    \end{enumerate}
\end{theorem}

As a corollary of Theorem \ref{thm:maintransverse}, we have the horospherical measure classification on $\Ga \ba G$, for relatively Anosov subgroups.

\begin{corollary} \label{cor:mainrelAnosov}
    Let $\Ga < G$ be a Zariski dense relatively Anosov subgroup. Let $\mu$ be a non-zero, $N$-invariant ergodic Radon measure on $\Ga \ba G$. Then either
    \begin{enumerate}
        \item $\mu$ is supported on $\mathcal{R}_{\Ga}$ and is a constant multiple of a Burger--Roblin measure, or 
        \item $\mu$ is supported on a closed $NM$-orbit in $(\Ga \ba G) \smallsetminus \mathcal{R}_{\Ga}$.
    \end{enumerate}
\end{corollary}
Note that the same holds for $NM$-invariant ergodic Radon measures by Theorem \ref{thm:maintransverse}.

\begin{remark}
    As we will see, we prove Theorem \ref{thm:maintransverse} for a product of general $\CAT(-1)$ spaces, when the vector-valued length spectrum is non-arithmetic, i.e., generates a dense additive subgroup. See Theorem \ref{thm:uniqueRadon} and Theorem \ref{thm:ergodicdiv}.
    
    To give a more concrete sense, we consider the following example. Let $\Gamma$ be a genus $3$ surface group acting simultaneously on two open disks $\tilde{\Sigma}, \tilde{\Sigma}'$ with pinched negative curvatures. Let $\Sigma$ and $\Sigma'$ be the resulting genus 3 surfaces. Then $\Gamma$ sits in $\Isom(\tilde{\Sigma} \times \tilde{\Sigma}')$ as a transverse and non-elementary subgroup. We claim that given $\Sigma$ and $\Sigma'$, one can slightly perturb the metrics on $\Sigma$ and $\Sigma'$ so that $\Gamma < \Isom(\tilde \Sigma \times \tilde \Sigma')$ admits non-arithmetic vector-valued length spectrum.
 
 Let $\gamma_1, \ldots, \gamma_4$ be disjoint simple closed curves on $\Sigma$ and let $\gamma'_{1}, \ldots, \gamma'_{4}$ be the corresponding ones on $\Sigma'$. Their lengths $\ell_{\Sigma}(\cdot)$ on $\Sigma$ and $\ell_{\Sigma'}(\cdot)$ on $\Sigma'$ are determined by the local choice of metrics on disjoint annular neighborhoods $A_1, \ldots, A_4$ and $A'_{1}, \ldots, A_{4}'$, as long as global CAT($-1$)-ness is guaranteed. Fixing the choices of metrics on $A_1$ and $A'_1$, thereby fixing a length vector $\mathbf{v}_{1} := (\ell_{\Sigma}(\gamma_1), \ell_{\Sigma'}(\gamma_{1}')) \in \R^2$, we perturb the metrics on $A_2$ and $A'_2$ slightly so that $\mathbf{v}_{2} := (\ell_{\Sigma}(\gamma_{2}), \ell_{\Sigma'}(\gamma_{2}')) \in \R^2$ is not commensurable to $\mathbf{v}_{1}$. If $\overline{\langle \mathbf{v}_{1}, \mathbf{v}_{2} \rangle}$ is the full $\mathbb{R}^{2}$, we can stop here. If not, we similarly perturb the metrics on $A_3$ and $A'_3$ so that $\mathbf{v}_{3} := (\ell_{\Sigma}(\gamma_{3}), \ell_{\Sigma'}(\gamma_{3}')) \in \R^2$ is not commensurable to the subgroup $\overline{\langle \mathbf{v}_{1}, \mathbf{v}_{2} \rangle}$. Do the same for $A_4$ and $A'_4$. The worst case is when we see subgroups isomorphic to $\Z$, $\Z^{2}$, $\Z \times \R$, or $\R^{2}$, and in every case we get a perturbed metric for which $\mathbf{v}_{1}, \ldots, \mathbf{v}_{4}$ generate a dense subgroup of $\R^{2}$. 
 
 Meanwhile, if $\tilde{\Sigma}$ and $\tilde{\Sigma}'$ are equipped with constant curvature $-1$, whence $\tilde{\Sigma} \times \tilde{\Sigma}'$ is a symmetric space $\H^2 \times \H^2$, then $\Gamma$ has non-arithmetic vector-valued length spectrum whenever $\Sigma$ and $\Sigma'$ are not isometric, as shown by Benoist \cite{Benoist2000proprietes} (see Theorem \ref{thm:nonarithmetic}). 
\end{remark}

\subsection{On the proof}

Our proof is rather geometric, and does not make use of any continuous flow on $\Ga \ba G$, such as one-dimensional diagonal flows given by $v \in \inte \fa^+$, or multi-dimensional action of $\exp \fa^+$. We also do not rely on the existence of Besicovitch-type covering.
These are major differences between our argument and previous literature, and enable us to classify horospherical invariant measures without restricting the supports of measures to smaller subsets.

More generally, we consider the product space $Z := \prod_{i = 1}^r X_i$, where $X_i$ is a proper geodesic $\CAT(-1)$ space, not necessarily a symmetric space for a Lie group. In this setting, the notion of \emph{transverse subgroup} $\Ga < \Isom(Z)$ is defined similarly. We then prove a measure classification 
for the $\Ga$-action on the horospherical foliation $\mathcal{H} := \partial Z \times \R^r$, where $\partial Z = \prod_{i = 1}^r \partial X_i$ and the $\Ga$-action on $\R^r$-component is given by Busemann cocycles for each $X_i$ componentwise. Then all results in the introduction are deduced from this.

The proof of this measure classification is based on extending the technique developed in our recent work \cite{CK_ML} to vector-valued cocycles in $\R^r$. In \cite{CK_ML}, so-called \emph{squeezing geodesics} were key players. While every geodesic in a $\CAT(-1)$ space is squeezing (Lemma \ref{lem:squeezingD}), it is no longer true in the product of $\CAT(-1)$ spaces due to the presence of flats. Our major technical difficulty lies in overcoming the presence of flats, by controling tuples of geodesics in each $X_i$'s and obtaining squeezing properties \emph{simultaneously} in each component.
We use geometric aspects of transverse subgroups for this.

We elaborate this further. Given a $\Ga$-invariant ergodic Radon measure $\mu$ on $\mathcal{H}$, we first show that for $\mu$-a.e. $(\xi, u) \in \mathcal{H} = \partial Z \times \R^r$, the point $\xi \in \partial Z$ is accumulated by a $\Ga$-orbit in $Z$, not just conically but ``fellow traveling'' the translates of the axis of a chosen loxodromic element of $\Ga$ in each component simultaneously (Theorem \ref{thm:radonCharge}). 
The ``fellow traveling'' property is based on the \emph{contracting} property of a geodesic in a $\CAT(-1)$ space which is weaker than squeezing, and we use the transverse property of $\Ga$ to guarantee the fellow traveling \emph{simultaneously} in each component.

Next, using the squeezing property of axes in each component \emph{simultaneoulsy}, we investigate the ``fellow traveling accumulations'' further and show that the measure $\mu$ is quasi-invariant under the translation by the vector-valued translation length in $\R^r$ of the chosen loxodromic element. Controlling this squeezing property and fellow traveling in each component simultaneously, we are able to precisely get the \emph{vector-valued} translation length (Theorem \ref{thm:trbytrlengthqi}). These compose the major step of the proof of our measure classification. We emphasize that we do not care about the ``speed'' of fellow traveling in each component, which might correspond to considering 1-dimensional diagonal flows.

\subsection{Organization} In Section \ref{section:prelim}, we present a brief review of the geometry of $\CAT(-1)$ spaces. We consider products of $\CAT(-1)$ spaces and prove simultaneous alignment property in Section \ref{sec:product}, which is one of the key observations in this paper. Section \ref{sec:UE} is devoted to the main rigidity result for measures on the horospherical foliations of product spaces. The ergodicity of such measures is proved in Section \ref{sec:ergodicity}. In Section \ref{section:homogeneous}, we consider higher-rank homogeneous spaces and deduce results stated in the introduction.

\subsection{Acknowledgements}
The authors would like to thank  Hee Oh for helpful conversations and useful comments on the earlier version of this paper. Kim extends his special gratitude to his Ph.D. advisor Hee Oh for her encouragement and guidance.

Choi was supported by the Mid-Career Researcher Program (RS-2023-00278510) through the National Research Foundation funded by the government of Korea, and by the KIAS individual grant (MG091901) at KIAS.

\subsection{Notation}
For reals $a, b, c$, we write the condition $|a-b| < c$ by $a=_{c} b$.

\section{Basic CAT(--1) geometry}\label{section:prelim}

In this section, we review basics of the geometry of CAT($-1)$ spaces. We refer the readers to classical references including \cite{gromov1987hyperbolic}, \cite{coornaert1990geometrie}, \cite{ghys1990bord}, and \cite{Bridson1999metric} for more details.

CAT($-1$) spaces are geodesic metric spaces where every geodesic triangle is no fatter than the corresponding comparison triangle in $\mathbb{H}^{2}$. Throughout this section, let  $(X, d)$ be a proper geodesic $\CAT(-1)$ space and let $x_{0} \in X$ be a basepoint. This forces that $X$ is \emph{uniquely geodesic}: for each $x, y \in X$, there exists a unique geodesic connecting $x$ to $y$, which we denote by $[x, y]$.

\subsection{Contracting property and squeezing property}\label{subsection:contracting}

We say that two geodesics $[x, y]$ and $[x', y']$ in $X$ are $C$-equivalent if $d(x, x') < C$ and $d(y, y') < C$. The $\CAT(-1)$ property implies the following: 

\begin{fact}[{\cite[Proposition 3.4.27]{ghys1990bord}}]\label{fact:CAT(-1)Thin}
Let $\gamma$ and $\gamma'$ be two compact geodesics that are $C$-equivalent. Then their Hausdorff distance is at most $C$.
\end{fact}

Given a geodesic $\gamma \subset X$ and a point $x \in X$, there exists the unique closest point on $\gamma$ from $x$. We denote that point by $\pi_{\gamma}(x)$. The map $\pi_{\gamma}(\cdot)$ is distance-decreasing, i.e., $1$-Lipschitz and  continuous. In fact, we have:

\begin{lem}[Contracting property]\label{lem:CAT(-1)Fellow}
Let $\ga \subset X$ be a geodesic and let $x, y \in X$ be such that $d(\pi_{\gamma}(x), \pi_{\gamma}(y))> 2$. Then there exist points $p, q \in [x, y]$ with $d(x, p) < d(x, q)$ such that
\begin{itemize}
\item $\diam(\pi_{\ga}([x, p]) \cup \{p\}) \le 2$,
\item $\diam(\pi_{\ga}([q, y]) \cup \{q\}) \le 2$, and
\item $[\pi_{\gamma}(x), \pi_{\gamma}(y)]$  and $[p, q]$ are $2$-equivalent. 
\end{itemize} 
\end{lem}

See Appendix \ref{appendix} for its proof.
As a consequence, any geodesic that is  far away from $\gamma$ cannot have large projection on $\gamma$. This is the so-called \emph{contracting property} of $\gamma$.

Up to changing the constant 2 above, this lemma follows from the classical tree approximations (\cite[Th{\'e}or{\`e}me 8.1]{coornaert1990geometrie}, \cite[Th{\'e}or{\`e}me 2.12]{ghys1990bord}). We give a proof in the appendix for completeness.

For every $x \in X$, every geodesic $\gamma \subset X$, and every $p \in \gamma$, the triangle $\triangle x \pi_{\gamma}(x) p$ is right-angled at $\pi_{\ga}(x)$. Hence, $\pi_{\gamma}(x)$ is $0.604$-close to $[x, p]$. This implies that:

\begin{cor}\label{cor:CAT(-1)Fellow}
Let $\gamma: \R \rightarrow  X$ be a geodesic, let $x \in X$ and let  $\gamma(t) = \pi_{\gamma}(x)$. Then for every $s \in \R$, we have 
\begin{equation}\label{eqn:approxDist}
d(x, \gamma(s)) =_{1.3} d(x, \gamma(t)) + |t-s|.
\end{equation}
\end{cor}

We now record a finer contracting behavior exhibited by geodesics in $X$, which we call the \emph{squeezing property}. See Figure \ref{figure.contractingsqueezing}.

\begin{lem}[Squeezing property]\label{lem:squeezingD}
Let $\gamma : \mathbb{R} \rightarrow X$ be a geodesic. Then for any $\epsilon>0$, there exists $L = L(\epsilon)>0$ such that for each $x, y \in X$ and $t \in \R$ with $\ga(t - a) =\pi_{\ga}(x)$ and $\ga(t + b)= \pi_{\ga}(y)$ for some $a, b \ge L$, we have
$$
d\big([x, y], \ga(t)\big) \le \epsilon.
$$
\end{lem}

In fact, geodesics in CAT($-1$) spaces enjoy even stronger \emph{exponentially squeezing property} thanks to the comparison principle. We leave the proof to interested readers.

As we will see later, squeezing geodesics are well-suited for studying horofunctions due to the following lemma.

\begin{lem}[{\cite[Lemma 5.6]{CK_ML}}]\label{lem:squeezing}
Let $\gamma : \R \to X$ be a geodesic. Fix $\epsilon >0$ and let $L = L(\epsilon) > 0$ as in Lemma \ref{lem:squeezingD}. Let $x_{1}, x_{2}, y_{1}$, and $y_{2}$ be points in $X$ and let $t \in \R$ be such that   \[
\pi_{\gamma}(x_{i}) \in \gamma \left((-\infty, t-L]\right)  \quad \text{and} \quad \pi_{\gamma}(y_{i}) \in\gamma\left( [t+L, +\infty)\right)  \quad \text{for }i = 1, 2.
\]
Then we have
 \[
d(x_{1}, y_{1}) - d(x_{1}, y_{2}) =_{8\epsilon} d(x_{2}, y_{1}) - d(x_{2}, y_{2}).
\]
\end{lem}

\begin{figure}[h]
\begin{tikzpicture}[scale=0.65]
    \draw[thick] (-7.5, 0) -- (7.5, 0);

    \draw[dashed, teal, thick] (-6, 2) .. controls (-5, 2) and (-3.5, 1) .. (-3.5, 0);
    \draw[dashed, teal, thick] (6.5, 3) .. controls (5, 3) and (3.5, 1) .. (3.5, 0);

    \filldraw[teal] (-3.5, 0) circle(2pt);
    \filldraw[teal] (3.5, 0) circle(2pt);

    \draw[red, thick] (-6, 2) .. controls (-3.5, 2)  and (-3.3, 0.3) .. (-2.7, 0.3) .. controls (-2.5, 0.2) and (2.5, 0.2) .. (2.8, 0.5) .. controls (3.3, 0.7) and (3.5, 3) .. (6.5, 3);

    \filldraw (-6, 2) circle(2pt);
    \filldraw (6.5, 3) circle(2pt); 

    \filldraw (-6, 2) node[left] {$x$};
    \filldraw (6.5, 3) node[right] {$y$};

    \filldraw (0, 0) circle(2pt);
    \filldraw[color=blue!80, opacity=0.3] (0, 0) circle(0.5);

    \draw[<->, teal, thick] (-3.5, -0.25) -- (0, -0.25);
        \draw[<->, teal, thick] (3.5, -0.25) -- (0, -0.25);

    \draw (2, -0.2) node[below] {\color{teal} $\ge L$};

    \draw (7.5, 0) node[right] {$\ga$};

\end{tikzpicture}
\caption{A squeezing geodesic $\gamma$} \label{figure.contractingsqueezing}
\end{figure}
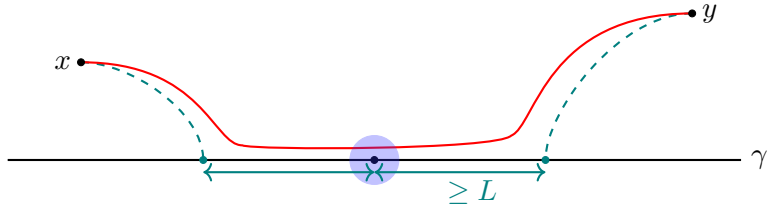

\subsection{Alignment}
We now define alignment between geodesics and points. 

\begin{definition}[Alignment]
Let $w, x, y, z \in X$. For a geodesic $[x, y] \subset X$ and $K \ge 0$, we say that the sequence $(w, [x, y])$ is \emph{$K$-aligned} if \[
d \big( \pi_{[x, y]}(w), \, x \big) < K.
\]
Similarly, we call that the sequence $([x, y], z)$ is \emph{$K$-aligned} if $(z, [y, x])$ is $K$-aligned.

Finally, we say that the sequence $(w, [x, y], z)$ is \emph{$K$-aligned} if both sequences $(w, [x, y])$ and $([x, y], z)$ are $K$-aligned. See Figure \ref{fig:alignment}.

\end{definition}

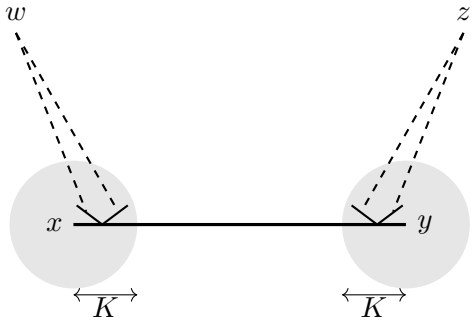
\begin{figure}[h]

\begin{tikzpicture}[scale=0.85]

\draw[very thick] (-2.6, 0) -- (2.6, 0);

\draw[dashed, thick] (-3.5, 3) -- (-1.9, 0.2);
\draw[dashed, thick] (-3.5, 3) -- (-2.4, 0.2);
\draw[thick] (-2.55, 0.3) -- (-2.15, 0) -- (-1.75, 0.3);

\begin{scope}[xscale=-1]

\draw[dashed, thick] (-3.5, 3) -- (-1.9, 0.2);
\draw[dashed, thick] (-3.5, 3) -- (-2.4, 0.2);
\draw[thick] (-2.55, 0.3) -- (-2.15, 0) -- (-1.75, 0.3);
\end{scope}

\draw(-2.9, 0) node {$x$};
\draw(2.9, 0) node {$y$};
\draw (-3.5, 3.3) node {$w$};
\draw (3.5, 3.3) node {$z$};

\fill[opacity=0.1] (-2.6, 0) circle (1);
\fill[opacity=0.1] (2.6, 0) circle (1);
\draw[<->] (-2.6, -1.1) -- (-1.6, -1.1);
\draw (-2.1, -1.3) node {$K$};

\begin{scope}[xscale=-1]
\draw[<->] (-2.6, -1.1) -- (-1.6, -1.1);
\draw (-2.1, -1.3) node {$K$};
\end{scope}
\end{tikzpicture}
\caption{Alignment of geodesics and points.}
\label{fig:alignment}
\end{figure}

 The following is immediate.

\begin{lem}\label{lem:alignDich}
    Let $\ga \subset X$ be a geodesic of length $L \ge 0$, let $0 \le D \le L$ and let $x \in X$. Then $(\gamma, x)$ is not $D$-aligned or $(x, \gamma)$ is not $(L-D)$-aligned.
\end{lem}

In general, we can define the alignment between compact geodesics and boundary points in the same way (see Definition \ref{dfn:alignHoro}). We first need the following fact. See Appendix \ref{appendix} for its proof.

\begin{lem}\label{lem:alignExtCts}
Let $\gamma \subset X$ be a compact geodesic. Then the nearest-point projection $\pi_{\gamma}(\cdot) : X \rightarrow \gamma$ extends continuously to the boundary $\partial X$. More explicitly, for every sequence $\{z_{n}\}_{n \in \N} \subset X$ converging to $z \in X \cup \partial X$, the limit $\lim_{n \to + \infty} \pi_{\gamma}(z_{n})$ exists. 
\end{lem}

Using this extended nearest-point projection, we can define the alignment between compact geodesics and boundary points.

\begin{definition}\label{dfn:alignHoro}

    Let $\xi \in \partial X$ and $\ga \subset X$ be a compact geodesic. For $K \ge 0$, we say that $(\xi, \ga)$ is \emph{$K$-aligned} if for every sequence $\{z_{i}\}_{i \in \N} \subset X$ converging to $\xi$, $(z_{i}, \gamma)$ is $K$-aligned eventually (i.e., for all large $i \in \N$). We define the alignment for $(\ga, \xi')$ and $(\xi, \ga, \xi')$ similarly for $\xi' \in X \cup \partial X$.
    
\end{definition}

\subsection{Shadows and alignment}

We make a useful elementary observation that the alignment can be interpreted in terms of shadows.

\begin{definition}
    For $x, y \in X$ and $R > 0$, we define the \emph{shadow} $O_R(x, y)$ of a ball of radius $R$ centered at $y$ viewed from $x$, as follows:
    $$
    O_R(x, y) := \{ w \in X \cup \partial X : d([x, w], y) < R \}.
    $$
\end{definition}

It is easy to see that for $x, y \in X$ and $R > 0$, if $\xi \in O_R(x, y)$, then
$$
d(x, y) - 2R \le \beta_{\xi}(x, y) \le d(x, y).
$$

We now interpret the alignment using shadows. First, note that  one can imagine that if $x, y, z, w \in X$ satisfy
$$
w \in O_R(x, y) \cap O_R(y, z),
$$
then $y$ comes earlier than $z$ along $[x, w]$. Indeed, one can show the following:

\begin{lemma} \label{lem:shadowandalignment}
    \
\begin{enumerate}
    \item For each $R > 1$ and $x, y, z, w \in X$, if  $w \in O_R(x, y) \cap O_R(y, z)$ holds, then
    $$
    (x, [y, z], w) \quad \text{is $6R$-aligned.}
    $$

    \item For each $R>1$ and $x, y, z, w \in X$, if $
    (x, [y, z], w)$ is $R$-aligned and $d(y, z) > 3R$, then 
    $$w \in O_{3R}(x, y)\cap O_{3R}(y, z).$$
\end{enumerate}
\end{lemma}

\subsection{Isometries}
We now turn to isometries of $X$. As a Gromov hyperbolic space, $X$ has the Gromov boundary $\partial X$. The isometries can be classified in terms of their fixed points in $X \cup \partial X$.
A non-trivial isometry $g \in \Isom(X)$ is either \emph{elliptic} (i.e., fixes a point in $X$), \emph{parabolic} (i.e., has a unique fixed point in $\partial X$), or \emph{loxodromic} (i.e., has a unique pair of two fixed points in $\partial X$). If $g \in \Isom(X)$ is of infinite order, it is either parabolic or loxodromic. 

Among them, a loxodromic element $g \in \Isom(X)$ preserves a unique geodesic $\gamma : \R \to X$ connecting two fixed points of $g$, called the \emph{axis} of $g$, and acts on it as a translation by $\tau_{g} > 0$. We call $\tau_g$ the \emph{translation length} of $g$.

Given a loxodromic $g \in \Isom(X)$, note that 
$$
\tau_g = \lim_{n \to +\infty} \frac{d(x, g^n x)}{n} > 0 \quad \text{for each } x \in X.
$$
Then we can observe the following:
$$
\tau_g  = \inf_{x \in X} d(x, g x) \quad \text{and} \quad \tau_{g^k} = |k| \tau_g \quad \text{for each } k \in \Z.
$$

Note that in the CAT($-1$) space $X$, every geodesic is squeezing (Lemma \ref{lem:squeezingD}) and hence every loxodromic isometry $g$ possesses a squeezing axis, which is unique up to reparametrization. Ideally, it is the most convenient to capture the squeezing property of $g$ in terms of the nearest-point projection onto the axis of $g$. However, the chosen basepoint $x_0 \in X$ might not be on the axis $\gamma$, and one often needs to relate the nearest-point projections onto $\gamma$ and $[x_0, g^k x_0 ]$ for various $k \in \Z$. The following lemma serves this purpose, whose proof can be found in \cite[Lemma 5.9]{CK_ML}.

\begin{lem}\label{lem:BGIPHeredi}
Let $g \in \Isom(X)$ be a loxodromic isometry, $\gamma: \R \rightarrow X$ its axis, and $x_{0} \in X$. Then there exists $C=C(g, \gamma, x_{0})>0$ such that the following holds. 
\begin{enumerate}
\item $d\left(g^{k} x_{0}, \gamma(\tau_{g} k)\right) < C$ for all $k \in \Z$.
\item Let $k \in \N$,  $x \in X$, and  $K \ge C$. Then 
$$
\left(x, [x_{0}, g^{k} x_{0}]\right) \text{ is not $K$-aligned} \quad \Longrightarrow \quad \pi_{\gamma}(x) \in  \gamma \left( [K-C, +\infty) \right).
$$
\item Let $k\in \N$, $x \in X$, and  $0 \le K \le \tau_{g} k - C$. Then
$$
\left(x, [x_{0}, g^{k} x_{0}]\right) \text{ is $K$-aligned} \quad \Longrightarrow \quad \pi_{\gamma}(x) \in \gamma \left( (-\infty,  K + C]\right).
$$
\end{enumerate}
Moreover, $C$ can be chosen so that $C(g, \ga, x_0) = C(g^{k}, \ga, x_0)$ for all $k \in \N$ and $C(g^{-1}, \widehat{\ga}, x_0) = C(g, \ga, x_0)$ where $\widehat{\ga}$ is the inversion of $\ga$.
\end{lem}
We often write $C(g) = C(g, \ga, x_0)$ by implicitly choosing its axis $\ga$.

\subsection{Non-elementary subgroups of isometries}

We call $\Ga < \Isom(X)$ \emph{discrete} if it acts properly on $X$. The class of subgroups of $\Isom(X)$ we are interested in is as follows:

\begin{definition} \label{def:noneltsubgp}
    A discrete subgroup $\Ga < \Isom(X)$ is called  \emph{non-elementary}~if
    \begin{itemize}
        \item $\Ga$ is not virtually cyclic, and
        \item $\Ga$ contains a loxodromic isometry.
    \end{itemize}
\end{definition}

We can characterize non-elementary subgroups in terms of their limit sets:

\begin{definition}
    Let $\Ga < \Isom(X)$ be a discrete subgroup. Its \emph{limit set} $\La(\Ga) \subset \partial X$ is the set of all accumulation points of $\Ga x \subset X$ on $\partial X$, for any fixed $x \in X$. One can see that $\La(\Ga)$ is compact and $\Ga$-invariant.
\end{definition}

Since $X$ is a Gromov hyperbolic space, the $\Ga$-action on $X \cup \partial X$ is a convergence action, and the limit set $\La(\Ga)$ is also the limit set as a convergence group. It is a fact that a discrete subgroup $\Ga < \Isom(X)$ is non-elementary if and only if $\# \La(\Ga) \ge 3$, and in this case the $\Ga$-action on $\La(\Ga)$ is minimal.

Given a loxodromic isometry $g \in \Isom(X)$, we denote by $g^{+}$ and $g^{-}$ the attracting and the repelling fixed points on the boundary $\partial X$ of $g$, respectively. We  say that two loxodromic isometries $g, h \in \Isom(X)$ are \emph{independent} if $\{g^+, g^{-} \}$ and $\{h^{+}, h^{-}\}$ are disjoint.

\begin{lem}\label{lem:indep}
Let $\Gamma < \Isom(X)$ be a non-elementary subgroup. For a loxodromic isometry $g \in \Ga$, there exists $h \in \Ga$ such that $hgh^{-1}$ and $g$ are independent. Moreover,  there are infinitely many pairwise independent loxodromic isometries in $\Ga$.
\end{lem}

The following is a variant of the so-called \emph{extension lemma} of Yang, which can be regarded as the coarse-geometric version of the Anosov closing lemma (cf. \cite[Lemma 3.8]{bowen2008equilibrium}).

\begin{lem}[Extension lemma {\cite[Lemma 1.13]{yang2019statistically}}]\label{lem:extension} \label{lem:extensionHoro}
Let $\Gamma < \Isom(X)$ be a non-elementary subgroup. Then for each loxodromic isometry $\varphi \in \Ga$, there exist $a_{1}, a_{2}, a_{3}\in \Gamma$ and $\alpha = \alpha(\varphi)>0$  such that  for each $x, y \in X \cup \partial X$, there exists $a\in \{a_{1}, a_{2}, a_{3}\}$ that  makes 
$$
(x, a \cdot [x_{0}, \varphi^{n}x_{0}], a \varphi^{n} a \cdot y) \quad \text{$\alpha$-aligned for all } n \in \N.
$$
Moreover, $\alpha$ can be chosen so that $\alpha(\varphi^k) = \alpha(\varphi)$ for all $k \in \Z$.
\end{lem}

The proof can be found in \cite[Lemma 5.12, Lemma 5.15]{CK_ML}.

\subsection{Horofunctions} \label{subsection:boundary}

We now discuss the boundaries of $X$. Recall that $X$ is  proper and CAT($-1$). Hence, its visual compactification, Gromov compactification and the horofunction compactification all coincide, i.e., 
$$
\partial_{vis} X =  \partial X = \partial^{h}X .
$$
In particular, for each $\xi \in \partial X$ the \emph{Busemann cocycle} $\beta_{\xi} : X \times X \to \R$ is well-defined: for every $x, y \in X$ and every sequence $\{z_n\}_{n \in \N} \subset X$ converging to $X$ in the Gromov compactification $X \cup \partial X$,
$$
\beta_{\xi} (x, y) := \lim_{n \to + \infty} d(x, z_n) - d(y, z_n)
$$
is well-defined. Furthermore, $\xi$ is \emph{visible}, i.e., the sequence $\{z_{n}\}_{n \in \N}$ for $\xi$ above can be taken along a geodesic.

We now give more detailed description of the horofunction compactification.
Let $\Lip^{1}(X)$ be the space of $\R$-valued 1-Lipschitz functions on $X$ and let $\Lip^{1}_{x_{0}}(X)$ be its subspace vanishing at the basepoint $x_{0} \in X$, i.e., \[\begin{aligned}
\Lip^{1}(X) &:= \{ f : X \rightarrow \mathbb{R} : \textrm{$f$ is $1$-Lipschitz}\},\\
\Lip^{1}_{x_{0}}(X) &:= \{ f \in \Lip^{1}(X) : f(x_{0}) = 0\},
\end{aligned}
\]
equipped with the compact-open topology. Here, $\Lip^{1}_{x_0}$ is closed in $\Lip^{1}(X)$.

Recall that $X$ is separable as it is given a proper metric. Therefore, $\Lip^{1}_{x_{0}} (X)$ is compact, Hausdorff, and second countable \cite[Proposition 3.1]{maher2018random}. Hence, it is completely metrizable and is  Polish. We identify $\Lip^1(X)$ and $\Lip_{x_0}^1(X) \times \R$ via the homeomorphism
\begin{equation} \label{eqn:lipandlip1}
f \in \Lip^1(X) \mapsto \left( f - f(x_{0}),  f(x_{0})\right).
\end{equation}

We also identify $\Lip_{x_0}^1(X)$ with the space of $\R$-valued 1-Lipschitz cocycles on $X$, i.e., $c : X \times X \to \R$ such that $\abs{c(x, y)} \le d(x, y)$ and $c(x, z) = c(x, y) + c(y, z)$ for all $x, y, z \in X$.
For each $f \in \Lip^1(X)$, we define the associated cocycle
$
\beta_f : X \times X \to \R
$
by
$$
\beta_f(x, y) = f(x) - f(y).
$$
Its restriction to $\Lip_{x_0}^1(X)$ gives the homeomorphism between $\Lip_{x_0}^1(X)$ and the space of all $\R$-valued continuous cocycles. Then the identifiaction $\Lip^1(X) \simeq \Lip_{x_0}^1(X) \times \R$ in Equation \eqref{eqn:lipandlip1} can be rephrased as
$$
f \mapsto (\beta_f, f(x_0)).
$$
The $\Isom(X)$-action on $\Lip^1(X)$ is now given as follows: for $g \in \Isom(X)$ and $f \in \Lip^1(X)$,
$$
g \cdot (\beta_f, f(x_0)) = (\beta_{g \cdot f}, f(x_0) + \beta_f (g^{-1}x_0, x_0)).
$$
Note that on the first component, which corresponds to $\Lip^1_{x_0}(X)$, we have $\beta_f \mapsto \beta_{g \cdot f}$.

There is a natural embedding $\iota : X \hookrightarrow \Lip_{x_0}^{1}(X)$, defined by \[
\iota : z \in X \quad \mapsto \quad  \left[ f_{z}( \cdot) := d(\cdot, z) - d(x_{0}, z) \right].
\] 
The closure of $\iota (X) \subset \Lip_{x_0}^{1}(X)$ is called the \emph{horofunction compactification} of $X$ and is denoted by $\overline{X}^{h}$. The complement $\overline{X}^{h} \smallsetminus \iota(X)$ is called the \emph{horofunction boundary} (or \emph{horoboundary}) of $X$  and is denoted by $\partial^{h} X$. As explained above, $\partial^{h} X$ is naturally identified with $\partial X$.

In terms of the identification $\Lip^{1}(X) \simeq \Lip_{x_{0}}^{1}(X) \times \mathbb{R}$, the subspace of $\Lip^{1}(X)$ corresponding to $\partial^{h} X$ is the space
\begin{equation} \label{eqn:horofol}
\mathcal{H} := \partial^{h} X \times \R,
\end{equation}
which is  $\Isom(X)$-invariant.

We call elements of $\partial^{h} X \times \mathbb{R}$ \emph{horofunctions}. They are 1-Lipschitz functions that are limits of sequences of the form $\{f_{z_{n}}( \cdot) + c_{n}\}_{n \in \N}$ for some $z_{n} \in X$ escaping to infinity and $c_{n} \in \R$.

Both $\partial^{h}X$ and $\mathcal{H} = \partial^{h} X \times \mathbb{R}$ are Polish. Hence, every locally finite Borel measure on these spaces is Radon, i.e., it is both inner and outer regular on Borel subsets.

\subsection{Conical limit sets}

We define conical limit sets using Busemann cocycles, which are also called radial limit sets. Fix a basepoint $x_0 \in X$, while the conical limit sets do not depend on the choice of the basepoint.

\begin{definition} \label{def:conical}
    Let $\Ga < \Isom(X)$ be a subgroup acting properly on $X$. A point $\xi \in \partial X$ is called a \emph{conical limit point} of $\Ga$ if there exist $K > 0$ and an infinite sequence $ \{g_n \}_{n \in \N} \subset \Ga $ such that
    $$
    \beta_{\xi}(x_0, g_n x_0) \ge d(x_0, g_n x_0) - K  \quad \text{for all } n \in \N.
    $$
    We denote the \emph{conical limit set} by $\La_c(\Ga) \subset \partial X$.
\end{definition}

Geometrically, $\xi$ is a conical limit point if and only if some (equivalently, every) geodesic ray $\gamma \subset X$ converging to $\xi$ has a $R$-neighborhood that contains infinitely many points in the $\Gamma$-orbit, for some $R > 0$. Equivalently, $\xi \in \La_{c}(\Ga)$ if and only if there exist $R > 0$ and a sequence $\{g_n\}_{n \in \N} \subset \Ga$ such that $\xi \in O_R(x_0, g_n x_0)$ for all $n \in \N$. The conical limit set $\La_c(\Ga)$ is $\Ga$-invariant.

\subsection{Guided limit sets}

In \cite{CK_ML}, we introduced the notion of guided and guided limit sets, which are variants of 
Coulon's contracting limit sets \cite{coulon2024patterson-sullivan} and Yang's $(L, \mathscr{F})$-limit sets \cite{yang2024conformal}.

\begin{definition} \label{def:guidedlimitset}
    Let $\Ga < \Isom(X)$ be a non-elementary subgroup. Let $\varphi \in \Ga$ be  a loxodromic isometry and let $C(\varphi) > 0$ be as in Lemma \ref{lem:BGIPHeredi} and fix $K\ge C(\varphi)$. We say that $\xi \in \partial X$ is a \emph{$(\varphi, K)$-guided limit point} of $\Ga$ 
if for each sufficiently large $n \in \N$, there exists $h \in \Gamma$ such that 
$$
(x_{0}, h[x_{0}, \varphi^{n} x_{0}], \xi) \quad \text{is $K$-aligned.}
$$
The collection of $(\varphi, K)$-guided limit points of $\Ga$ called the \emph{$(\varphi, K)$-guided limit set} of $\Gamma$. We denote it by $\La_{\varphi, K}(\Ga)$.

\end{definition}

The role of $K$ in the definition of $(\varphi, K)$-guided limit set is quite flexible:

\begin{lem}[{\cite[Lemma 6.4]{CK_ML}}]\label{lem:squeezedInv}
    Let $\Ga < \Isom(X)$ be a non-elementary subgroup. Let $\varphi \in \Ga$ be  a loxodromic isometry and let $C = C(\varphi) > 0$ be as in Lemma \ref{lem:BGIPHeredi}.  Then for each $K>C$, 
    $$
    \La_{\varphi, K}(\Ga) = \La_{\varphi, C}(\Ga).
    $$
     Moreover, $\La_{\varphi, C}(\Ga)$ is $\Gamma$-invariant.
\end{lem}

 For a non-elementary subgroup $\Ga < \Isom(X)$, isometries $g, \varphi \in \Ga$, constants $C > 0$, and $n \in \N$, we set
$$
U_C ( g; \varphi, n) := \left\{ \xi \in \La(\Ga) : \textrm{$(x_{0}, g[x_{0}, \varphi^{n} x_{0}], \xi)$ is $C$-aligned}\right\}.
$$
In \cite{CK_ML}, we observed that they form a basis for the  topology on the guided limit set.

\begin{lem}[{\cite[Lemma 7.9]{CK_ML}}]\label{lem:nbdBasis}
    Let $\Ga < \Isom(X)$ be a non-elementary subgroup containing a loxodromic isometry $\varphi \in \Ga$, and let $C=C(\varphi) > 0$ be as in Lemma \ref{lem:BGIPHeredi}. Then 
    $$\{U_{C}(g; \varphi, n) : g \in \Gamma, n \in \N \}$$ 
    forms a basis for the topology of $\La_{\varphi, C}(\Ga) \subset \partial X$.

In other words, for each  $\xi \in \La_{\varphi, C}(\Ga)$, for each open set $O \subset \partial X$ with $\xi \in O$ and for each $N \in \N$, there exist $g \in \Ga$, $n > N$, and an open set $V \subset \partial X$ such that
\[
\xi \in V  \cap \Lambda (\Gamma ) \subset U_{C}(g; \varphi, n) \subset O.
\]
\end{lem}

\section{Product spaces} \label{sec:product}

We now consider a product of $\CAT(-1)$ spaces.
Let  $X_1, \dots, X_r$ be proper geodesic $\CAT(-1)$ spaces. Abusing notations, we use the same notation $d$ for the metric on each $X_i$. We consider the product space
$$
Z := X_1 \times \cdots \times X_r
$$
and set its boundary as
$$
\partial Z := \partial X_1 \times \cdots \times \partial X_r.
$$
One can see that $\partial Z$ is not the same as the geometric boundary of $Z$.
We define the convergence of sequences in $Z$ to $\partial Z$ as follows:

\begin{definition} \label{def:convergence}
    We say that a sequence $\{z_n = (x_{1, n}, \dots, x_{r, n}) \}_{n \in \N} \subset Z$ converges to $\xi = (\xi_1, \dots, \xi_r) \in \partial Z$ if for each $1 \le i \le r$,
    $$
    x_{i, n} \to \xi_i \quad \text{as } n \to + \infty.
    $$
    In this case, we also write $z_n \to \xi$.
\end{definition}

We also set
$$
\Isom(Z) := \Isom(X_1) \times \cdots \times \Isom(X_r).
$$
We call $\Ga < \Isom(Z)$ discrete if its action on $Z$ is proper.
With the above notion of convergence, we also define the limit set on $\partial Z$.

\begin{definition} \label{def:limitproduct}
    Let $\Ga < \Isom(Z)$ be a discrete subgroup. The \emph{limit set} $\La(\Ga) \subset \partial Z$ of $\Ga$ is the set of all accumulation points of a $\Ga$-orbit in $Z$, in the sense of convergence defined in Definition \ref{def:convergence}. One can see that $\La(\Ga)$ is a compact $\Ga$-invariant subset of $\partial Z$. 
\end{definition}

In this product case, we consider \emph{vector-valued Busemann cocycles}. For $\xi = (\xi_1, \dots, \xi_r) \in \partial Z$ and $z = (x_1, \dots, x_r), z' = (x_1', \dots, x_r') \in Z$, we set
\begin{equation}  \label{eqn:prodBuse}
\beta_{\xi}(z, z') := \left( \beta_{\xi_1}(x_1, x_1'), \dots, \beta_{\xi_r}(x_r, x_r') \right).
\end{equation}
Similarly, we consider the \emph{vector-valued distance}
\begin{equation}  \label{eqn:proddistance}
\kappa(z, z') := \left( d(x_1, x_1'), \dots, d(x_r, x_r') \right).
\end{equation}
For simplicity, we also use the notation 
$$
\beta_{\xi}^{i}(z, z') := \beta_{\xi_i}(x_i, x_i') \quad \text{and} \quad d_i (z, z') := d(x_i, x_i') \quad \text{for } 1 \le i \le r.
$$

\subsection{Transverse subgroups}

We mainly consider discrete subgroups with certain transversality.

\begin{definition} \label{def:transprod}
    We say that $\Ga < \Isom(Z)$ is \emph{transverse} if
    \begin{itemize}
        \item (divergent) for any infinite sequence $\{g_n\}_{n \in \N} \subset \Ga$ and any fixed $z \in Z$, we have for each $1 \le i \le r$ that
        $$
         d_i(g_n z, z) \to + \infty \quad \text{as } n \to + \infty,
        $$
        and
        \item (antipodal) for any distinct $\xi = (\xi_1, \dots, \xi_r), \zeta = (\zeta_1, \dots, \zeta_r) \in \La(\Ga)$,
        $$
        \xi_i \neq \zeta_i \quad \text{for all } 1 \le i \le r.
        $$
    \end{itemize}
\end{definition}

Note that for a transverse subgroup $\Ga < \Isom(Z)$, each of its projection $\Ga_i < \Isom(X_i)$ is a discrete subgroup. Moreover, its limit set $\La(\Ga_i) \subset \partial X_i$ is the same as the projection of $\La(\Ga) \subset \partial Z$. The following is an easy observation.

\begin{lemma} \label{lem:homeoprojection}
    Let $\Ga < \Isom(Z)$ be a transverse subgroup. Then for each $1 \le i \le r$, the projection $\La(\Ga) \to \La(\Ga_i)$ is an equivariant homeomorphism.
\end{lemma}

\begin{proof}
Equivariance is clear. So it remains to prove that the projection is injective. This is a direct consequence of the antipodality.
\end{proof}

Since each projection $\Ga_i < \Isom(X_i)$ acts on $X_i \cup \partial X_i$ as a convergence group with the limit set $\La(\Ga_i) \subset \partial X_i$, we have the following corollary:

\begin{corollary} \label{cor:convergenceaction}
    Let $\Ga < \Isom(Z)$ be a transverse subgroup. Then the $\Ga$-action on $\La(\Ga)$ is a convergence action.
\end{corollary}

Lemma \ref{lem:homeoprojection} also induces a $\Ga$-equivariant homeomorphism
$$
\La(\Ga_i) \to \La(\Ga_j) \quad \text{for each } i, j = 1, \dots, r.
$$
This yields componentwise ``type-preserving'' phenomenon for transverse subgroups.

\begin{corollary} \label{cor:typepreserving}
    Let $\Ga < \Isom(Z)$ be a transverse subgroup. Then each projection $\Ga \to \Ga_i < \Isom(X_i)$ has a finite kernel. Moreover, for any $(g_1, \dots, g_r) \in \Ga$, if $g_i \in \Isom(X_i)$ is loxodromic for some $1 \le i \le r$, then $g_j \in \Isom(X_j)$ is loxodromic for all $1 \le j \le r$.
\end{corollary}

We now define the non-elementary property.

\begin{definition}
    We say that a transverse subgroup $\Ga < \Isom(Z)$ is \emph{non-elementary} if $\# \La(\Ga) \ge 3$.
\end{definition}

By Corollary \ref{cor:convergenceaction}, non-elementary transverse subgroup $\Ga < \Isom(Z)$ acts minimally on $\La(\Ga) \subset \partial Z$.

\subsection{Simultaneous alignment}

For each $1 \le i \le r$, let $x_i, x'_i \in X_i$ and let $\ga_i \subset X_i$ be a geodesic. Writing tuples $z = (x_1, \dots, x_r), z' = (x_1', \dots, x_r') \in Z$ and $\ga = (\ga_1, \dots, \ga_r)$, we say that
$$
(z, \ga, z') \quad \text{is \emph{$K$-aligned} for $K \ge 0$}
$$
if $(x_i, \ga_i, x_i')$ is $K$-aligned for all $1 \le i \le r$. We also write
$$
[z, z'] := ( [x_1, x_1'], \dots, [x_r, x_r']).
$$

Divergence and antipodality in the definition of transverse groups imply that the projections of a $\Gamma$-orbit to different factors are somehow synchronized. For example, the divergence implies the following.

\begin{prop}\label{prop:divFactors}
Let $\Gamma < \Isom(Z)$ be a transverse subgroup. Let $z =(x_1, \ldots, x_r ) \in Z$. Then for each $R>0$, there exists $R' = R'(R, z)>0$ such that, for every $(g_1, \ldots, g_r ) \in \Gamma$ with $d(x_1, g_1 x_1) > R'$, we have $d(x_i, g_i x_i ) > R$ for each $1 \le i \le r$. 
\end{prop}

\begin{proof}
For simplicity, we suppose $r = 2$. Suppose to the contrary that there exists a sequence $\{g_{n} = (g_{1, n}, g_{2, n})\}_{n \in \N} \subset \Gamma$ such that $d(x_2, g_{2, n} x_{2} ) \le R$ but $d(x_1, g_{1, n} x_{1}) > n$ for all $n \in \N$. Then $\{g_{n}\}_{n \in \N}$ is indeed an infinite sequence but the projection of the orbit on $X_{2}$ does not diverge. This contradicts the divergence condition.  
\end{proof}

A key observation in this paper is that for a transverse subgroup, alignment occurs simultaneously at each component. By the interpretation of alignment using shadows given in Lemma \ref{lem:shadowandalignment}, this is a consequence of the following:

\begin{proposition} \label{prop:componentwiseshadow}
    Let $\Ga < \Isom(Z)$ be a transverse subgroup. Let $z = (x_1, \dots, x_r) \in Z$. Then for any $R > 0$, there exists $R' = R'(R, z) > 0$ such that if $g = (g_1, \dots, g_r), h = (h_1, \dots, h_r) \in \Ga$ satisfy
    $
    h_1 x_1 \in O_R (x_1, g_1 x_1)
    $, then 
    $$
    h_i x_i \in O_{R'}(x_i, g_i x_i) \quad \text{for all } 1 \le i \le r.
    $$
\end{proposition}

\begin{proof}
    For simplicity, we assume $r = 2$. Suppose to the contrary that there exist sequences $\{ (g_{1, n}, g_{2, n}) \}_{n \in \N}, \{ (h_{1, n}, h_{2, n}) \}_{n \in \N} \subset \Ga$ such that 
$$
h_{1, n} x_1 \in O_R( x_1, g_{1, n} x_1) \quad \text{and} \quad h_{2, n} x_2 \notin O_n (x_2, g_{2, n} x_2) \quad \text{for all } n \in \N.
$$
Then for all $n \in \N$,
$$
g_{1, n}^{-1}h_{1, n} x_1 \in O_R( g_{1, n}^{-1} x_1,  x_1) \quad \text{and} \quad g_{2, n}^{-1}h_{2, n} x_2 \notin O_n (g_{2, n}^{-1} x_2, x_2).
$$
In particular, both $\{g_{2, n}^{-1} \}_{n \in \N}$ and $\{g_{2, n}^{-1} h_{2, n}\}_{n \in \N}$ are infinite sequences.
This implies that, after passing to a subsequence,
$$
\lim_{n \to + \infty} g_{1, n}^{-1} h_{1, n} x_1 \neq \lim_{n \to + \infty} g_{1, n}^{-1} x_1 \quad \text{and} \quad \lim_{n \to + \infty} g_{2, n}^{-1} h_{2, n} x_2 = \lim_{n \to + \infty} g_{2, n}^{-1} x_2.
$$
Note that both
$$
\lim_{n \to + \infty} (g_{1, n}^{-1} h_{1, n}, g_{2, n}^{-1} h_{2, n})(x_1, x_2) \quad \text{and} \quad \lim_{n \to + \infty} (g_{1, n}^{-1}, g_{2, n}^{-1})(x_1, x_2)
$$
are points in $\La(\Ga)$.
However, their first components are different while their second components are the same. This contradicts  the antipodality.
\end{proof}

Proposition \ref{prop:componentwiseshadow} and Lemma \ref{lem:shadowandalignment} say that once we have an alignment on one component, we have it for all other components.

\begin{proposition} \label{prop:simultaneousalign}
    Let $\Ga < \Isom(Z)$ be a transverse subgroup and $z = (x_1, \dots, x_r) \in Z$. Then for any $K > 0$, there exists $\widehat{C} = \widehat{C}(K, z) > 0$ such that if $g = (g_1, \dots, g_r), h = (h_1, \dots, h_r), k = (k_1, \dots, k_r) \in \Ga$ satisfy that $(x_1, [g_1 x_1, h_1 x_1], k_1 x_1)$ is $K$-aligned, then
    $$
    (z, [gz, hz], kz) \quad \text{is $\widehat{C}$-aligned.}
    $$
\end{proposition}

\begin{proof}
For simplicity, assume that $r = 2$. Fix $z = (x_1, x_2) \in Z$ and $K > 1$. Then by Lemma \ref{lem:shadowandalignment} and Proposition \ref{prop:componentwiseshadow}, there exists $\widehat{C} > 0$ so that if $(g_1, g_2), (h_1, h_2), (k_1, k_2)  \in \Ga$ satisfy that $(x_1, [g_1 x_1, h_1 x_1], k_1 x_1)$ is $K$-aligned and $d(g_1 x_1, h_1 x_1) > 3K$, then $(z, [gz, hz], kz) $ is $\widehat{C}$-aligned.

Now by Corollary \ref{cor:typepreserving}, $\# \{ (g_1, g_2) \in \Ga : d(x_1, g_1x_1) \le 3K \} < + \infty$. Hence, we can take $\widehat{C} > 0$ large enough so that  if $(g_1, g_2), (h_1, h_2), (k_1, k_2)  \in \Ga$ satisfy  $d(g_1 x_1, h_1 x_1) \le 3K$, then $(z, [gz, hz], kz) $ is $\widehat{C}$-aligned.
\end{proof}

We are now ready to define the subset of $\partial Z$ that captures the dynamics of $\Gamma$. We first define shadows in $Z$. For $R > 0$ and $z = (x_1, \dots, x_r), z' = (x_1', \dots, x_r') \in Z$, we set
$$
O_R(z, z') := \prod_{i = 1}^r O_R(x_i, x_i') \subset Z \cup \partial Z.
$$

\begin{definition}\label{dfn:conical}
Let $\Gamma < \Isom(Z)$ be a transverse subgroup. 
We define the \emph{conical limit set} $\La_{c}(\Ga) \subset \partial Z$ by
$$
\La_{c}(\Ga) := \left\{ \xi \in \partial Z : \begin{matrix}
    \exists R > 0, z \in Z, \text{ an infinite sequence } \{g_n\}_{n \in \N}\subset \Ga \\
    \text{s.t. } \xi \in O_R(z, g_n z) \text{ for all } n \in \N.
\end{matrix}
\right\}
$$
\end{definition}

Proposition \ref{prop:componentwiseshadow}  says that $\La_{c}(\Ga)$ is precisely the homeomorphic preimage of  $\La_{c}(\Ga_i) \subset \partial X_i$ under the homeomorphism $\La(\Ga) \to \La(\Ga_i)$ in Lemma \ref{lem:homeoprojection}. This is again the same as the conical limit set of $\Ga$, for its convergence action on $\La(\Ga)$.

\subsection{Patterson--Sullivan theory}

We revisit the Patterson--Sullivan theory for this product space $Z$. In this generality, one can use recent theory of Blayac--Canary--Zhu--Zimmer \cite{BCZZ_PS}. As our Busemann cocycles take vector values, a choice of linear form $\psi : \R^r \to \R$ is involved in defining conformal density. We fix a basepoint $z_0 \in Z$.

\begin{definition}
    Let $\Ga < \Isom(Z)$ be a subgroup. For $\delta \ge 0$ and a linear form $\psi :~\R^r \to \R$, a family of Borel measures $\{ \nu_z \}_{z \in Z}$ on $\La(\Ga) \subset \partial Z$ is called a \emph{$\delta$-dimensional $\psi$-conformal density} of $\Ga$ if 
    \begin{itemize}
\item {\rm ($\Gamma$-invariance)} for every $g \in \Ga$ and $z \in Z$,
$$g_{*} \nu_{z} = \nu_{gz},$$
\item {\rm (conformality)} for every $z, w \in Z$, two measures $\nu_z$ and $\nu_w$ are in the same class and \[
\frac{d\nu_{z}}{d\nu_{w}}(\xi) = e^{-\delta \cdot \psi( \beta_{\xi}(z, w))} \quad \textrm{a.e., and}
\]
\item {\rm (normalization)} $\nu_{z_0}(\partial Z) = 1$.
\end{itemize}
\end{definition}

Similarly, we also choose a linear form to define a Poincar\'e series of $\Ga$: for a linear form $\psi : \R^r \to \R$ and $s \in \R$,
$$
\mathcal{P}_{\Ga, \psi}(s) := \sum_{g \in \Ga} e^{-s \psi(\kappa(z_0, g z_0))}.
$$
The associated \emph{critical exponent} is defined as 
$$
\delta_{\psi}(\Ga) := \inf \{ s > 0 : \mathcal{P}_{\Ga, \psi}(s) < + \infty \} \in [0, + \infty].
$$
\begin{definition} \label{def:divtypegpandmeasure}
    We say that a transverse subgroup $\Ga < \Isom(Z)$ is of \emph{$\psi$-divergence type} if $\delta_{\psi}(\Ga) < + \infty$ and $\mathcal{P}_{\Ga, \psi}(\delta_{\psi}(\Ga)) = + \infty$. 
    We also say that a conformal density $\nu$ of $\Ga$ is of \emph{divergence type} if $\Ga$ is of $\psi$-divergence type where $\psi$ is a  linear form associated to $\nu$.
\end{definition}

As a special case of results in \cite{BCZZ_PS}, we obtain the following. 
Theorems stated below were proved in (\cite{CZZ_transverse}, \cite{KOW_PD}) when each $X_i$ is a rank one Riemannian symmetric space.

\begin{theorem}[{\cite[Theorem 4.1]{BCZZ_PS}}] 
    Let $\Ga < \Isom(Z)$ be a non-elementary transverse subgroup and let $\psi : \R^r \to \R$ be a linear form. If $\delta_{\psi}(\Ga) < + \infty$, then there exists a $\delta_{\psi}(\Ga)$-dimensional $\psi$-conformal density of $\Ga$.
\end{theorem}

Indeed, existence of conformal density is equivalent to finiteness of the critical exponent. The following was proved for transverse subgroups of Lie groups, but the same proof works in our setting.

\begin{theorem}[{\cite[Proposition 10.1]{BCZZ_2}}] \label{thm:BCZZ_proper}
Let $\Ga < \Isom(Z)$ be a non-elementary transverse subgroup and let $\psi : \R^r \to \R$ be a linear form. If there exists a $\delta$-dimensional $\psi$-conformal density of $\Ga$, then 
$$
\delta_{\psi}(\Ga) \le \delta.
$$
In particular, $\delta_{\psi}(\Ga) < + \infty$.
\end{theorem}

When $\delta_{\psi}(\Ga) < +\infty$, we have $\psi(\kappa(z_0, g_n z_0)) \to + \infty$ for any infinite sequence $\{g_n \}_{n \in \N} \subset \Ga$. Then a classical construction of ``Schottky subgroup'' of $\Ga < \Isom(Z)$ implies $\delta_{\psi}(\Ga) > 0$ as well.

As part of their generalization of Hopf--Tsuji--Sullivan dichotomy, Blayac--Canary--Zhu--Zimmer proved the following:

\begin{theorem}[{\cite[Theorem 1.3]{BCZZ_PS}}] \label{thm:BCZZHTS}
    Let $\Ga < \Isom(Z)$ be a non-elementary transverse subgroup and $\{\nu_z\}_{z \in Z}$ a $\delta$-dimensional $\psi$-conformal density of $\Ga$, for a linear form $\psi : \R^r \to \R$. 
    
    Then the following are equivalent:
    \begin{enumerate}
    \item $\delta = \delta_{\psi}(\Ga) < + \infty$ and $\Ga$ is of $\psi$-divergence type.
    \item the conical limit set $\La_{c}(\Ga)$ is $\nu_z$-conull for all $z \in Z$.\end{enumerate}
    Moreover, in this case, the $\Ga$-action on $(\La(\Ga), \nu_z)$ is ergodic for all $z \in Z$.
\end{theorem}

In fact, the conical limit set considered by Blayac--Canary--Zhu--Zimmeer has a slightly different form, because their result is for an arbitrary convergence group, not necessarily induced by an isometric action, and they introduced shadows defined intrinsically to the convergence group action. We first describe their shadows in our setting. For a non-elementary transverse subgroup $\Ga < \Isom(Z)$, noting that $\Ga$ acts on $\La(\Ga)$ as a convergence group, fix a metric $\mathsf{d}$ on the compactification $\Ga \cup \La(\Ga)$ \cite[Proposition 2.3]{BCZZ_PS}. Then for $\epsilon > 0$ and $\ga \in \Ga$, the set
$$
\mathcal{S}_{\epsilon}(\ga) := \ga ( \La(\Ga) \smallsetminus B_{\epsilon}(\ga^{-1}))
$$
is the associated shadow defined in \cite{BCZZ_PS}, where $B_{\epsilon}(\ga^{-1})$ denotes the open ball centered at $\ga^{-1}$ of radius $\epsilon$ with respect to the metric $\mathsf{d}$. Then they showed that a point $\xi \in \La(\Ga)$ is a conical limit point in the sense of convergence action if and only if there exists $\epsilon > 0$ and an infinite sequence $\{\ga_n\}_{n \in \N} \subset \Ga$ such that $\xi \in \mathcal{S}_{\epsilon}(\ga_n)$ for all $n \in \N$ \cite[Lemma 5.4]{BCZZ_PS}.

To apply their results to our setting, as in Theorem \ref{thm:BCZZHTS}, we record the following comparability of their shadows and the shadows we consider.

\begin{lemma}
    Let $\Ga < \Isom(Z)$ be a non-elementary transverse subgroup and $z \in Z$.
    \begin{enumerate}
        \item For any $\epsilon > 0$, there exists $R = R(\epsilon, z) > 0$ such that 
        $$
\mathcal{S}_{\epsilon}(\ga) \subset O_R(z, \ga z) \quad \text{for all } \ga \in \Ga.
        $$

        \item For any $R > 0$, there exists $\epsilon = \epsilon(R, z) > 0$ such that 
        $$
 O_R(z, \ga z) \cap \La(\Ga) \subset \mathcal{S}_{\epsilon}(\ga)  \quad \text{for all } \ga \in \Ga.
        $$
    \end{enumerate}

\end{lemma}

\begin{proof}
We first show (1). Suppose to the contrary that for some $\epsilon > 0$, there exist sequences $\{ \ga_n \}_{n \in \N} \subset \Ga$ and $\{\xi_n\}_{n \in \N} \subset \partial Z$ such that
$$
\xi_n \in \mathcal{S}_{\epsilon}(\ga_n) \smallsetminus O_n(z, \ga_n z) \quad \text{for all } n \in \N.
$$
Here, the sequence $\{\ga_n\}_{n \in \N}$ must be infinite.
For each $n \in \N$ we have
$$
\ga_n^{-1} \xi_n \notin  B_{\epsilon}(\ga_n^{-1}) \cup  O_n(\ga_n^{-1} z, z).
$$
Since the $\Ga$-action on $Z \cup \partial Z$ is convergence action (Corollary \ref{cor:convergenceaction}), after passing to a subsequence, there exists $\xi \in \La(\Ga)$ so that $\ga_n^{-1} \to \xi$ in the compactifiaction $\Ga \cup \La(\Ga)$ and $\ga_n^{-1} z \to \xi$ in $Z \cup \partial Z$. Since $\ga_n^{-1} \xi_n \notin O_n(\ga_n^{-1} z, z)$ for all $n \in \N$, we have $\ga_n^{-1} \xi_n \to \xi$ as well. On the other hand, $B_{\epsilon/2}(\xi) \subset B_{\epsilon}(\ga_n^{-1})$ for all large $n \in \N$, and hence this contradicts that $\ga_n^{-1} \xi_n \notin B_{\epsilon}(\ga_n^{-1})$ for all $n \in \N$.

To see (2), suppose that for some $R > 0$, there exist sequences $\{ \ga_n \}_{n \in \N} \subset \Ga$ and $\{\xi_n\}_{n \in \N} \subset \La(\Ga)$ such that
$$
\xi_n \in O_R(z, \ga_n z) \smallsetminus \mathcal{S}_{1/n}(\ga_n) \quad \text{for all } n \in \N.
$$
Again, $\{\ga_n\}_{n \in \N}$ is an infinite sequence, and we have that for each $n \in N$,
$$
\ga_n^{-1} \xi_n \in O_R(\ga_n^{-1} z,  z) \cap B_{1/n}(\ga_n^{-1}) \quad \text{for all } n \in \N.
$$
After passing to a subsequence, we denote by $\xi \in \La(\Ga)$ the limit of sequences $\{ \ga_n^{-1} z \}_{n \in \N}$ and $\{\ga_n^{-1}\}_{n \in \N}$. Since $\ga_n^{-1} \xi_n \in O_R(\ga_n^{-1} z,  z)$ for all $n \in \N$, we have $\lim_{n \to + \infty} \ga_n^{-1} \xi_n \neq \xi$ after passing to a subsequence. On the other hand, this contradicts that $\ga_n^{-1} \xi_n \in B_{1/n}(\ga_n^{-1})$ for all $n \in \N$.
\end{proof}

\section{Rigidity of ergodic invariant Radon measures} \label{sec:UE}

We continue the setting of Section \ref{sec:product}.
In this section, we prove a measure rigidity on horospherical foliations.

In the rest of this section, we fix a basepoint $z_0 \in Z$.
The \emph{horospherical foliation} of $Z$ is the space
\begin{equation} \label{eqn:horofoliation}
\mathcal{H} := \partial Z \times \R^r
\end{equation}
and $\Isom(Z)$ acts on $\mathcal{H}$ as follows: for $g \in \Isom(Z)$ and $(\xi, u) \in \mathcal{H}$,
$$
g\cdot (\xi, u) := (g \xi, u + \beta_{\xi}(g^{-1} z_0, z_0)).
$$

We define a Radon measure on $\mathcal{H}$ as follows:

\begin{definition} \label{def:candidateergodicmeasure}
    Let $\Ga < \Isom(Z)$ be a non-elementary transverse subgroup and $\nu := \{ \nu_z \}_{z \in Z}$ be a $\delta_{\psi}(\Ga)$-dimensional $\psi$-conformal density of $\Ga$, for a linear form $\psi : \R^r \to \R$. We define a Radon measure $\mu_{\nu}$ on $\mathcal{H} = \partial Z \times \R^r$ by 
    $$
    d\mu_{\nu}(\xi, u) := e^{\delta_{\psi}(\Ga) \cdot \psi(u)} \cdot d \nu_{z_0} (\xi) \, du
    $$
    where $du$ is the Lebesgue measure on $\R^r$. If $\Ga$ is of $\psi$-divergence type, then we write
    $$
    \mu_{\psi} := \mu_{\nu}.
    $$
\end{definition}

\begin{remark}
It follows from the conformality of $\nu$ that $\mu_{\nu}$ is $\Ga$-invariant. If $\Ga$ is of $\psi$-divergence type,  then  there exists a unique $\delta_{\psi}(\Ga)$-dimensional $\psi$-conformal density of $\Ga$ by Theorem \ref{thm:BCZZHTS}. This is a reason for writing $\mu_{\psi} = \mu_{\nu}$ in this case. Moreover, by Theorem \ref{thm:BCZZHTS}, $\mu_{\psi}$ is supported on $\La_{c}(\Ga) \times \R^r$.
\end{remark}

To present the precise statement of our rigidity result, we also consider the following notion for the distribution of translation lengths of loxodromic elements. We say that an element $g = (g_1, \dots, g_r) \in \Isom(Z)$ is \emph{loxodromic} if $g_i \in \Isom(X_i)$ is loxodromic for all $1 \le i \le r$. In this case, we write its vector-valued translation length as
$$
\tau_g := (\tau_{g_1}, \dots, \tau_{g_r}) \in \R^r.
$$

\begin{definition} \label{def:spectrum}
    For $\Ga < \Isom(Z)$, its (vector-valued) \emph{length spectrum} is defined as 
    $$
    \Spec(\Ga) := \{ \tau_g \in \R^r : g \in \Ga \text{ is loxodromic.}\}
    $$
    We say that $\Spec(\Ga)$ is \emph{non-arithmetic} if it generates a dense additive subgroup of $\R^r$.
\end{definition}

\subsection{Rigidity of measures}

The following is our main rigidity theorem.

\begin{theorem} \label{thm:uniqueRadon}
Let $\Ga < \Isom(Z)$ be a non-elementary transverse subgroup with non-arithmetic length spectrum. Suppose that there exists a $\Ga$-invariant ergodic Radon measure $\mu$ on $\mathcal{H}$.
\begin{enumerate}
    \item If $\mu$ is supported on $\La_c(\Ga) \times \R^r$, then $\Ga$ is of $\psi$-divergence type for some linear form $\psi : \R^r \to \R$ and
    $$
    \mu \text{ is a constant multiple of $\mu_{\psi}$.}
    $$
    \item If $\mu$ is supported on $\mathcal{H} \smallsetminus (\La(\Ga) \times \R^r$), then $\mu$ is a constant multiple of 
    $$
    \sum_{g \in \Ga} D_{g \cdot \xi} \quad \text{for some $\xi \in \mathcal{H} \smallsetminus (\La(\Ga) \times \R^r)$}
    $$
    where $D_{g \cdot \xi}$ is the Dirac measure at $g \cdot \xi$.
\end{enumerate}

\end{theorem}

The rest of this section is devoted to the proof of Theorem \ref{thm:uniqueRadon}. 
We prove the theorem by establishing a robust relation between invariant Radon measures and guided limit sets. Note that due to ergodic decompositions, Theorem \ref{thm:uniqueRadon} can be regarded as the classification of $\Ga$-invariant Radon measures on $\mathcal{H}$. 

\subsection{Concentration on guided limit sets}
We first show that invariant ergodic Radon measures on $\mathcal{H}$ are charged on guided limit sets. Let 
$$\Psi : \La(\Ga_1) \to \La(\Ga)$$ 
be the $\Ga$-equivariant homeomorphism give in Lemma \ref{lem:homeoprojection}.
For a loxodromic $\varphi \in \Ga$ and $C > 0$, denote by $\varphi_i \in \Ga_i < \Isom(X_i)$ the $i$-th component of $\varphi \in \Ga$ and set 
$$
\La_{\varphi, C}(\Ga) := \Psi(\La_{\varphi_1, C}(\Ga_1)).
$$

\begin{theorem}\label{thm:radonCharge}

Let $\Ga < \Isom(Z)$ be a non-elementary transverse subgroup, let $\varphi \in \Ga$ be a loxodromic element, and let $C = C(\varphi_1)$ be as in Lemma \ref{lem:BGIPHeredi}. 
Let $\mu$ be a $\Ga$-invariant ergodic Radon measure on $\mathcal{H}$ supported on $\La_{c}(\Ga) \times~\R^r$. Then the measure $\mu$ is supported on
$$
 \La_{\varphi, C}(\Ga) \times \R^r \subset \mathcal{H}.
$$

\end{theorem}

\begin{proof}
    Applying Lemma \ref{lem:extension} to $\varphi_1 \in \Ga_1 <  \Isom(X_1)$, we get $\alpha(\varphi_1) > 0$ and $a_{1}, a_{2}, a_{3} \in \Gamma$ whose first components satisfy the conclusion of Lemma \ref{lem:extension} for $\varphi_1$ and $\Ga_1 < \Isom(X_1)$. Let $ C(\varphi_1) > 0$ be as in Lemma \ref{lem:BGIPHeredi} for $g=\varphi_1$. We set $C_0 := 10(\alpha(\varphi_1) +C(\varphi_1))$. 

For each $K>0$ let 
\[
\La_{K} :=  \left\{ \xi \in \partial Z : \begin{matrix}
\exists \text{ an infinite sequence } \{g_j\}_{j \in \N} \subset \Ga \text{ s.t.}\\
\beta_{\xi}^1(z_0, g_{j} z_{0}) \ge d_1(z_{0}, g_{j}z_{0}) - K \text{ for all } j \in \N
\end{matrix}  \right\}.
\]
Then $\Gamma \cdot (\La_K \times \R^r) \subset \mathcal{H}$ is  $\Gamma$-invariant. Moreover,
 \[
\Lambda_{c} (\Gamma) \times \mathbb{R}^r = \bigcup_{K > 0} \Gamma \cdot (\La_K \times \R^r) 
\]
since $\La_{c}(\Ga) = \Psi(\La_{c}(\Ga_1))$ by Proposition \ref{prop:componentwiseshadow}.
Since $\Lambda_c (\Gamma) \times \R^r$ has positive $\mu$-value,
$$\Gamma \cdot (\La_K \times \R^r) \quad \text{has positive $\mu$-value for all large $K > 0$.}
$$
We fix such $K > 100 C_{0} + 2 \sum_{i = 1}^r \sum_{j = 1}^3 d_i(z_0, a_j z_0)$. Then it follows from the $\Ga$-invariance of $\mu$ that $\mu(\La_K \times \R^r)  > 0$. 
For each $R > 0$, we set 
$$\mathcal{H}_{K, R} := \La_K \times [-R, R]^r.$$ 
Since $\La_K \times \R^r = \cup_{R=1}^{\infty} \mathcal{H}_{K, R}$, $$
\mu(\mathcal{H}_{K, R}) > 0 \quad \text{for all large } R > 0.
$$
We fix such $R>0$. 

Now we pick $n>\frac{100(C_{0}+K+1)}{\min_{i}\tau_{\varphi_i}}$ and $k > 0$. We define a map
$$
F = F_{n, k}  : \mathcal{H}_{K, R} \to \mathcal{H}
$$
as follows. For each $\Xi = (\xi, u) \in \mathcal{H}_{K, R}$, there exists $g \in \Gamma$ such that 
\begin{equation} \label{eqn:defofgXi}
d_1(z_0, gz_0) > k \quad \text{and} \quad \beta_{\xi}^1(z_0, g z_0) \ge d_1(z_0, g z_0) - K.
\end{equation} Among many such $g$'s, take the one with minimal  $d_1(z_0, g z_0)$ and call it $g_{\Xi}$.\footnote{There exists a technicality when several candidates tie. An easy rescue is to first enumerate $\Gamma = \{g^{(1)}, g^{(2)}, \ldots\}$, and we choose the earliest whenever there is a tie.} Then the map $\Xi \in \mathcal{H}_{K, R} \mapsto g_{\Xi}$ is Borel measurable.
By Lemma \ref{lem:extensionHoro}, there exists $a_{\Xi} \in \{a_{1}, a_{2}, a_{3}\}$ such that \footnote{Again, when more than one of $\{a_1, a_2, a_3\}$ do the job we choose the earliest.}
\begin{equation} \label{eqn:defofFmap}
    \begin{matrix}
        \text{the first component of}\\
        \left(z_0, g_{\Xi} \cdot a_{\Xi} [z_0, \varphi^{n} z_{0}], g_{\Xi} \cdot a_{\Xi} \varphi_1^{n} a_{\Xi} \cdot g_{\Xi}^{-1} \xi \right)\\
        \text{is $C_{0}$-aligned}.
    \end{matrix}
\end{equation}
 This map $\Xi \mapsto a_{\Xi}$ is also Borel measurable. 
We now set 
 \[
F (\Xi) := g_{\Xi} \cdot a_{\Xi} \varphi^{n} a_{\Xi} \cdot g_{\Xi}^{-1} \Xi.
\]
Let \[
D := 100\left(C_{0} +  n \cdot \max_{i} \tau_{\varphi_i} + \sum_{i = 1}^{r} \sum_{j=1}^{3} d_i(z_0, a_j z_0) \right).
\]

By \cite[Claim in the proof of Theorem 7.5]{CK_ML}, 
\begin{equation} \label{eqn:finitetoone}
F \text{ is at most } 3 \cdot \# \{ g \in \Gamma : d_1(z_0, g z_0) \le D\}\text{-to-one}.
\end{equation}
We simply write $M := 3 \cdot \# \{ g \in \Gamma : d_1(z_0, g z_0) \le D\}$, which is finite by Corollary \ref{cor:typepreserving}. Then we have
\[\begin{aligned}
\mu(F(\mathcal{H}_{K, R})) &= \mu \left( \bigcup_{g \in \Gamma, a \in \{a_1, a_2, a_3\}} F \left( \{\Xi \in \mathcal{H}_{K, R} : g_{\Xi} = g, a_{\Xi} = a\} \right) \right) \\
&\ge \frac{1}{M}\sum_{g \in \Gamma, a \in \{a_1, a_2, a_3\}} \mu \left( F \left( \{\Xi \in \mathcal{H}_{K, R} : g_{\Xi} = g, a_{\Xi} = a\} \right) \right) \\
&=  \frac{1}{M}\sum_{g \in \Gamma, a \in \{a_1, a_2, a_3\}} \mu \left( ga\varphi^n a g^{-1} \{\Xi \in \mathcal{H}_{K, R} : g_{\Xi} = g, a_{\Xi} = a\} \right) \\
&= \frac{1}{M}\sum_{g \in \Gamma, a \in \{a_1, a_2, a_3\}} \mu \left( \{\Xi \in \mathcal{H}_{K, R} : g_{\Xi} = g, a_{\Xi} = a\} \right) \\
&= \frac{1}{M}\mu(\mathcal{H}_{K, R}).
\end{aligned}
\]

Now to see the image of $F$, let $\Xi = (\xi, u) \in \mathcal{H}_{K, R}$. For simplifity, write $g := g_{\Xi}$ and $a := a_{\Xi}$. Then
$$
F(\Xi) = ( ga \varphi^n a g^{-1} \xi, u + \beta_{\xi} ( (g a \varphi^n a g^{-1})^{-1} z_0, z_0) )
$$
Fixing a sequence $\{z_j\}_{j \in \N} \subset \Ga z_0 \subset Z$ converging to  $ga \varphi^n a g^{-1} \xi \in \partial Z$, we have 
$$\begin{aligned}
\beta_{\xi} ( (g a \varphi^n a g^{-1})^{-1} z_0, z_0) & = \beta_{g a \varphi^n a g^{-1}\xi}( z_0, g a \varphi^n a g^{-1} z_0) \\
& = \lim_{j \to + \infty} \kappa(z_0, z_j ) - \kappa( g a \varphi^n a g^{-1} z_0, z_j).
\end{aligned}
$$
By Equation \eqref{eqn:defofFmap} and Proposition \ref{prop:simultaneousalign}, there exists $\widehat{C} = \widehat{C}(C_0, z_0) > C_0 + 2$ such that 
$$
\left(z_0, ga [z_0, \varphi^{n} z_{0}], z_j \right) \quad \text{is $\widehat{C}$-aligned for all large $j \in \N$}.
$$
Lemma \ref{lem:CAT(-1)Fellow} then tells us that, for each large $j$, there exist $p, q \in [z_0, z_j]$ with $p$ coming first (as tuples of points) such that $d_i(p, gaz_0) \le \widehat{C} + 2$ and $d_i(q, g a \varphi^n z_0) \le \widehat{C} + 2$. 
It follows that for each $1 \le i \le r$ and all sufficiently large $j \in \N$,
$$\begin{aligned}
\beta_{\xi}^i ( (g a \varphi^n a g^{-1})^{-1} z_0, z_0) & =_{15\widehat{C}} d_i(z_0, g a z_0) + d_i(z_0, \varphi^{n} z_0) + d_i(g a \varphi^n z_0, z_j) \\
& \qquad \quad - d_i( g a \varphi^n a g^{-1} z_0, z_j) \\
& =_{\widehat{C}} d_i(z_0, g a z_0) + d_i(z_0, \varphi^{n} z_0) \\
& \qquad \quad + \beta_{g a \varphi^n a g^{-1}\xi}^i ( g a \varphi^n z_0, g a \varphi^n a g^{-1} z_0) \\
& =_{d_i(z_0, a z_0)} d_i(z_0, g z_0) +  d_i(z_0, \varphi^{n} z_0)   \\
& \qquad \quad + \beta_{\xi}^{i}(g a^{-1} z_0, g z_0) + \beta_{\xi}^{i}(g z_0, z_0) \\
& =_{d_i(z_0, a z_0)} d_i(z_0, g z_0) +  d_i(z_0, \varphi^{n} z_0) +  \beta_{\xi}^{i}(g z_0, z_0).
\end{aligned}
$$
By Equation \eqref{eqn:defofgXi} and Proposition \ref{prop:componentwiseshadow}, there exists $\widehat{K} = \widehat{K}(K, z_0)$ such that
$$
d_i(z_0, g z_0) - \widehat{K} \le \beta_{\xi}^{i} (z_0, g z_0) \le d_i(z_0, g z_0) \quad \text{for all } 1 \le i \le r.
$$
Hence, setting $\widehat{D} := D + 16 \widehat{C} + \widehat{K} + \max_i d_i(z_0, \varphi^{n} z_0)$, we have 
$$
\abs{\beta_{\xi}^i ( (g a \varphi^n a g^{-1})^{-1} z_0, z_0)} \le \widehat{D} \quad \text{for all } 1 \le i \le r.
$$
Therefore,
$$
u + \beta_{\xi} ( (g a \varphi^n a g^{-1})^{-1} z_0, z_0) \in \left[-R - \widehat{D}, R + \widehat{D} \right]^r.
$$
In addition, by Equation \eqref{eqn:defofFmap}, we have $d_1(z_0, ga z_0) > k - \sum_{j = 1}^3 d_1(z_0, a_j z_0)$ and that the first component of $(z_0, ga [z_0, \varphi^n z_0], g a \varphi^n a g^{-1}\xi)$ is $C_{0}$-aligned.

This implies that $F(\mathcal{H}_{K, R})$ is contained in
 \[
B_{k;n} := \left\{ (\zeta, v) \in \mathcal{H} : 
\begin{matrix}
v \in \left[-R - \widehat{D}, R+ \widehat{D} \right]^r \text{ and } \exists h \in \Ga \text{ such that}\\
d_1(z_0, h z_0) > k - \sum_{j = 1}^3 d_1(z_0, a_j z_0) \text{ and} \\ 
\text{first component of } (z_0, h [z_0, \varphi^n z_0], \zeta) \text{ is $C_{0}$-aligned}
\end{matrix}\right\}.
\]
Hence, we have
$$
\mu (B_{k;n}) \ge \mu(\mathcal{H}_{K, R})/M > 0.
$$
Note that the set $B_{k;n}$ is decreasing in $k$. Since $\mu$ is a Radon measure and $B_{k;n} \subset \partial Z \times \left[-R - \widehat{D}, R + \widehat{D}\right]^r$ which is \emph{compact}, we have $\mu(B_{k;n}) < + \infty$. Therefore, setting
\begin{equation} \label{eqn:defofBn}
B_n := \bigcap_{k > 0} B_{k;n},
\end{equation}
we have
\begin{equation} \label{eqn:positiveBn}
\mu(B_n) = \lim_{k \to + \infty} \mu(B_{k;n}) \ge \mu(\mathcal{H}_{K, R})/M > 0,
\end{equation}
noting that $M$ does not depend on $k$.

Now, $\Gamma \cdot B_{n}$ is a $\Gamma$-invariant set of positive $\mu$-measure. Hence, by the $\Ga$-ergodicity of $\mu$, we have that $\Ga \cdot B_n$ is $\mu$-conull, and therefore
$$
\bigcap_{n} \Ga \cdot B_n \quad \text{is $\mu$-conull}.
$$
We then show that for each $(\zeta, v) \in \bigcap_{n} \Ga \cdot B_n $, we have $\zeta \in \La_{\varphi, C_{0}+2}(\Ga)$. This finishes the proof by Lemma \ref{lem:squeezedInv}.

Let $(\zeta, v) \in \bigcap_{n} \Ga \cdot B_n$. Then for each large enough $n \in \N$, there exists  $h_0 \in \Ga$ so that
the first component of
$(z_0, h [z_0, \varphi^n z_0], h_0^{-1}\zeta)$ is $C_{0}$-aligned for infinitly many $h \in \Ga$.
In other words,
$$\begin{matrix}
\text{the first component of } ( h_0 z_0, h_0 h [z_0, \varphi^n z_0], \zeta)\\
\text{is $C_{0}$-aligned for infinitely many } h \in \Ga.
\end{matrix}
$$
Among infinitely many such $h \in \Ga$, we can choose one such that
$$d_1(h_0 z_0, h_0 h [z_0, \varphi^n z_0]) > d_1(z_0, h_0 z_0) + 2$$
and hence 
$$
d_1([h_0 z_0, z_0], h_0 h [z_0, \varphi^n z_0]) > 2.
$$
Lemma \ref{lem:CAT(-1)Fellow} tells us that $\pi_{h_0 h [x_0, \varphi^n z_0 ]}([h_0 z_0, z_0 ])$ has diameter at most 2. Therefore,
$$
\text{the first component of }
( z_0, h_0 h [z_0, \varphi^n z_0], \zeta) \quad \text{is $(C_{0}+2)$-aligned.}
$$
Since this holds for all large $n \in \N$, we conclude $\zeta \in \La_{\varphi, C_{0} + 2}(\Ga)$.
\end{proof}

\subsection{Quasi-invariance under translations} \label{subsec:trqi}

For $a \in \R^r$, consider a map $T_a : \mathcal{H} \to \mathcal{H}$ given by 
$
(\xi, u) \mapsto (\xi, u + a)
$.
For a Radon measure $\mu$ on $\mathcal{H}$, we consider its pullback  measure $T_a^*\mu$: for each Borel subset $E \subset \mathcal{H}$,
$$T_a^* \mu (E) := \mu(T_a E).$$
For a loxodromic  $g \in \Isom(Z)$, we simply write $T_{g} := T_{\tau_g}$.
We show that invariant ergodic measures on $\mathcal{H}$ are quasi-invariant under this translation.

\begin{theorem} \label{thm:trbytrlengthqi}
    Let $\Ga < \Isom(Z)$ be a non-elementary transverse subgroup. Let $\mu$ be a $\Ga$-invariant ergodic Radon measure on $\mathcal{H}$ supported on $\La_c(\Ga) \times \R^r$. Then for a loxodromic $\varphi \in \Ga$, there exists $\la \ge 0$ such that
    $$
    \frac{ d T_{\varphi}^* \mu}{d \mu} = e^{\la} \quad \text{a.e.}
    $$

\end{theorem}

\begin{proof}
Let $\varphi \in \Ga$ be a loxodromic element and let $C= C(\varphi) > 0$ be the constant satisfying Lemma \ref{lem:BGIPHeredi} for each component $\varphi_i \in \Isom(X_i)$, with the choice of axis $\ga_i : \R \to X_i$. As in \cite[Proof of Theorem 7.10]{CK_ML}, we may assume that 
$$
\tau_{\varphi_1} > 100 \widehat{C}
$$
where $\widehat{C} = \widehat{C}(C, z_0) > C$ is the constant given in Proposition \ref{prop:simultaneousalign}.

We first aim to show that 
\begin{equation} \label{eqn:trabscont}
(T_{\varphi}^{*} \nu)(E) \ge \nu(E)
\end{equation}
for each Borel subset $E \subset \mathcal{H}$. Note that by Theorem \ref{thm:radonCharge}, $\mu$ is supported on $\La_{\varphi, C}(\Ga) \times \R^r$.

\medskip
{\bf \noindent Step 1.} First consider the case that $E = K \times I$ for a compact subset $K \subset \La_{\varphi, C}(\Ga)$ and a compact box $I \subset \R^r$. 

We fix some open subset $O \subset \La(\Ga)$ such that $K \subset O$ and $\epsilon > 0$. Let $L = L(0.001\epsilon)>0$ be as in Lemma \ref{lem:squeezing} for $\gamma_1 $. 

Recall that $\La(\Ga) = \Psi(\La(\Ga_1))$. For $h = (h_1, \dots, h_r) \in \Ga$ and $n \in \N$, we simply write
$$
U_C(h; \varphi, n) := \Psi(U_C(h_1; \varphi_1, n)).
$$

Recall that  $\widehat{C} = \widehat{C}(C, z_0) > C$ is the constant given in Proposition \ref{prop:simultaneousalign}.
By Lemma \ref{lem:nbdBasis}, for each $\xi \in K$, there exist $g(\xi) \in \Gamma$ and $n(\xi) > \frac{2L+100 \widehat{C}}{\min_i \tau_{\varphi_i}}  + 4$ such that 
  \[
\xi \in U_{C} \left(g(\xi); \varphi, n(\xi) \right) \subset O.
\]
Let $\mathcal{U} := \left\{ U_C \left(g(\xi); \varphi, n(\xi)\right) : \xi \in K \right\}$, which is a countable collection of sets. For convenience, let us enumerate $\mathcal{U}$ based on their $d_1$-distances from $z_0$, i.e, let \[
\mathcal{U} = \{U_{1}, U_{2}, \ldots\}
\]
where 
$
U_j := U_C (g_j; \varphi, n_j)
$
for each $j \in \N$
so that 
 \[
d_1(z_0, g_1 \varphi^{n_1} z_0) \le 
d_1(z_0, g_2 \varphi^{n_2} z_0) \le \cdots.
\]

We will now define a subcollection 
$$\mathcal{V} := \{U_{i(1)}, U_{i(2)}, \ldots\} \subset \mathcal{U}$$ by inductively defining $i(1), i(2), \ldots$. We let $i(1) = 1$. Now, having defined $i(1), \ldots, i(N)$, define $i(N+1)$ as the smallest $j \in \N$ such that $U_{j}$ is disjoint from $U_{i(1)} \cup \cdots \cup U_{i(N)}$.

For each $l \in \N$, we set 
\begin{equation} \label{eqn:defofCl}
C_{l} := U_{i(l)} \cup \bigcup \left\{ U_{k} : k \ge i(l), U_{k} \cap U_{i(l)} \neq \emptyset \right\}.
\end{equation}
Then $\{C_{l} : l \in \N \}$ is a covering of $K$ contained in $O$.

Via the homeomorphism $\Psi : \La(\Ga_1) \to \La(\Ga)$, it follows from \cite[First claim in the proof of Theorem 7.10]{CK_ML} that 
for each $l \in \N$, 
\begin{equation} \label{eqn:trqiclaimClinU}
    C_{l} \subset U_C \left(g_{i(l)}; \varphi, n_{i(l)} - 1\right).
\end{equation}

Now for each  $l \in \N$, we  define a map $F_{l} : C_{l} \times I \rightarrow  \mathcal{H}$ as follows: for $g= g_{i(l)}$, we set 
\begin{equation} \label{eqn:defofFl}
F_{l} : \Xi \mapsto g \varphi g^{-1} \Xi.
\end{equation}
Then we have $\mu \left( F_{l}(C_{l} \times I) \right) = \mu(C_l \times I)$ as $\mu$ is $\Gamma$-invariant. 

\begin{claim*}
    We have 
    \begin{equation} \label{eqn:claimforFl}
    F_{l} (C_{l}\times I) \subset  U_{i(l)} \times \text{$(\epsilon$-neighborhood of $I + \tau_{\varphi})$}.
    \end{equation}
\end{claim*}
To see this, we simply write $g = g_{i(l)}$ and $n = n_{i(l)} - 1$. We then fix $\Xi = (\xi, u) \in C_l \times I$. Note that
$$
F_l (\Xi) = (g \varphi g^{-1} \xi, u + \beta_{\xi}(g \varphi^{-1} g^{-1} z_0, z_0)).
$$
The inclusion for the first component is due to \cite[Second claim in the proof of Theorem 7.10]{CK_ML}. Hence, we now show the inclusion for the second component.

For the second component, it suffices to show 
\begin{equation} \label{eqn:trqiclaim2goal2}
\abs{\beta_{\xi}^i(g \varphi^{-1} g^{-1} z_0, z_0) - \tau_{\varphi_i}} < \epsilon \quad \text{for all } 1 \le i \le r.
\end{equation} 
Let $\{z_j\}_{j \in \N} \subset \Ga z_0 \subset Z$ be a sequence converging to $\xi$.
Then
$$
\beta_{\xi}(g \varphi^{-1} g^{-1} z_0, z_0) = \lim_{j \to + \infty} \kappa(g \varphi^{-1} g^{-1} z_0, z_j) - \kappa(z_0,  z_j).
$$
By Equation \eqref{eqn:trqiclaimClinU} and Proposition \ref{prop:simultaneousalign},
$$
(z_0, g [z_0, \varphi^{n} z_0 ], z_j ) \quad \text{is $\widehat{C}$-aligned for all large $j \in \N$.}
$$

In the rest of this proof, write $\ga = (\ga_1, \dots, \ga_r)$ and consider the nearest-point projection and parametrization of $\ga$ componentwisely. Then for all large $j \in \N$, it follows from Lemma \ref{lem:BGIPHeredi}(3) that
\begin{equation} \label{eqn:trqiclaim2secondcomp}
    \pi_{g \gamma} (z_0) \subset g \gamma \left( \left(-\infty, 2\widehat{C} \right] \right) \quad \text{and} \quad \pi_{g \gamma} (z_{j}) \subset g \gamma \left( \left[ n \tau_{\varphi} - 2\widehat{C}, +\infty \right) \right).
\end{equation}
Since  $n \cdot \min_i \tau_{\varphi_i} - 4 \widehat{C} > 2L $ and each component geodesic of $g \ga$ is squeezing  (Lemma \ref{lem:squeezing}), there exists $p \in [z_0, z_j]$ such that
$$
\norm{\kappa(p, g\ga(n \tau_{\varphi}/2))}_{\infty} \le 0.001 \epsilon.
$$

Meanwihle, note that $\left( g\varphi^{-1} g^{-1} z_0, g [z_0, \varphi^{n} z_0] \right)$ is also $\widehat{C}$-aligned; otherwise, one component of $\pi_{g \ga}(g \varphi^{-1} g^{-1} z_0)$ belongs to $g \ga([0, + \infty))$ by Lemma \ref{lem:BGIPHeredi}(2), and therefore one component of $\pi_{g \ga}(z_0)$ is contained in  $ g \ga([\tau_{\varphi}, + \infty))$ which contradicts Equation \eqref{eqn:trqiclaim2secondcomp}. Hence, it follows from Lemma \ref{lem:BGIPHeredi}(3) that 
$$
\pi_{g\gamma} (g\varphi^{-1} g^{-1} z_0) \subset g\gamma \left( \left(-\infty, 2 \widehat{C} \right] \right). 
$$
Together with Equation \eqref{eqn:trqiclaim2secondcomp} and $n \tau_{\varphi_i} - 4 \widehat{C} > 2 L + 2\tau_{\varphi_i}$ for all $1 \le i \le r$, the squeezing property of each component geodesic of $g \ga$ implies that there exist  $q_{1},q_{2}\in [g\varphi g^{-1} z_0, z_j]$, with $q_1$ coming earlier than $q_2$, such that \[
\norm{\kappa\left(q_1, g \gamma(n\tau_{\varphi}/2 - \tau_{\varphi})\right)}_{\infty}, \norm{\kappa\left(q_2, g \gamma(n\tau_{\varphi}/2)\right)}_{\infty} < 0.001\epsilon.
\]
 
Now we have for each $1 \le i \le r$ that
\[
\begin{aligned}
d_i(g\varphi^{-1} g^{-1} z_0, z_j) - d_i(z_0, z_j) 
&= \left( d_i(g\varphi^{-1} g^{-1}z_0, q_1) + d_i(q_1, q_2) + d_i(q_2, z_j) \right) \\
& \quad - \left( d_i(z_0, p) + d_i(p, z_j) \right) \\
&=_{0.006\epsilon} d_i\left(g\varphi^{-1} g^{-1}z_0,  g\gamma(n\tau_{\varphi}/2 - \tau_{\varphi})\right) \\
& \qquad \quad + d_i\left(  g\gamma(n\tau_{\varphi}/2 - \tau_{\varphi}),   g\gamma(n\tau_{\varphi}/2)\right) \\
& \qquad \quad + d_i\left( g \gamma(n\tau_{\varphi}/2), z_j\right) \\
& \qquad \quad - d_i\left(z_{0},   g \gamma(n\tau_{\varphi}/2)\right) - d_i\left( g \gamma(n\tau_{\varphi}/2), z_j \right) \\
&= d_i\left(  g\gamma(n\tau_{\varphi}/2 - \tau_{\varphi}),   g\gamma(n\tau_{\varphi}/2)\right)  = \tau_{\varphi_i}.
\end{aligned}
\]
Taking the limit $j \to + \infty$, Equation \eqref{eqn:trqiclaim2goal2} follows. This completes the proof of the claim.

\medskip

Now by the above claim and disjointness of $U_{i(l)}$'s,  we have
 \[\begin{aligned}
\mu( O \times (\textrm{$\epsilon$-neighborhood of $I+\tau_{\varphi}$})) &\ge \mu\left( \bigcup_{l} F_l (C_{l} \times I) \right) \\
&= \sum_{l} \mu \left( F_l (C_l \times I)\right) \\
&= \sum_{l} \mu (C_{l} \times I) \\
& \ge \mu(K \times I).
\end{aligned}
\]
Note that $\mu( O \times (\textrm{$\epsilon$-neighborhood of $I+\tau_{\varphi}$})) < + \infty$ since $\mu$ is Radon.
Since $\epsilon > 0$ and an open set $O \supset K$ are arbitrary, we have 
$$
(T_{\varphi}^* \mu)(K \times I) = \mu(K \times (I + \tau_{\varphi})) \ge \mu(K \times I). 
$$

\medskip
{\bf \noindent Step 2.} Consider the case that $E = A \times B$ for Borel $A \subset \partial Z$ and a box $B \subset \mathbb{R}^r$. Since $\mu$ is supported on $\La_{\varphi, C}(\Ga) \times \R^r$, we may assume that $A \subset \La_{\varphi, C}(\Ga)$. By the inner regularity of $\mu$ and $T_{\varphi}^* \mu$, there exist compact subsets $E_1, E_2 \subset E$ such that
$$
\abs{\mu(E) - \mu(E_1)} < \epsilon \quad \text{and} \quad \abs{(T_{\varphi}^*\mu)(E) - (T_{\varphi}^*\mu)(E_2)  } < \epsilon.
$$
Considering projections of $E_1 \cup E_2$ to $A$ and $B$, we obtain compact subsets $K \subset A$ and $I \subset B$ so that 
$$
\abs{\mu(E) - \mu(K \times I)} < \epsilon \quad \text{and} \quad \abs{(T_{\varphi}^*\mu)(E) - (T_{\varphi}^*\mu)(K \times I)  } < \epsilon.
$$
Since $B$ is a box, we can take the smallest box containing $I$ and hence we may assume that $I$ is a compact box. Applying Step 1 to $K \times I$, we have
$$
(T_{\varphi}^*\mu)(E) \ge \mu(E) - 2 \epsilon.
$$
Since $\epsilon > 0$ is arbitrary, $(T_{\varphi}^*\mu)(E) \ge \mu(E)$ follows.

\medskip
{\bf \noindent Step 3.} When $E \subset \mathcal{H}$ is a finite union of open sets of the form $O_1 \times O_2$ for open sets $O_1 \subset \partial Z$ and open boxes $O_2 \subset \R^r$,  $E$ is a disjoint union of finitely many Borel subsets of the form $A \times B$, where $A \subset \partial Z$ is Borel and $B \subset \R^r$ is a box. Hence, $(T_{\varphi}^*\mu)(E) \ge \mu(E)$ follows from Step 2.

\medskip
{\bf \noindent Step 4.} When $E\subset \mathcal{H}$ is an open set, $E$ is a countable union of open sets of the form  $O_1 \times O_2$ for open sets $O_1 \subset \partial Z$ and open boxes $O_2 \subset \R^r$. Hence,  $(T_{\varphi}^*\mu)(E) \ge \mu(E)$ follows from Step 3.

\medskip
{\bf \noindent Step 5.} Finally, suppose that $E \subset \mathcal{H}$ is a Borel subset. Then it follows from Step 4 and the outer regularity of $\mu$ and $T_{\varphi}^*\mu$ that
$$
(T_{\varphi}^*\mu)(E) \ge \mu(E).
$$

\medskip
Now we have shown Equation \eqref{eqn:trabscont}, and hence $\mu$ is absolutely continuous with respect to $T_{\varphi}^* \mu$. Since both $\mu$ and $T_{\varphi}^*\mu$ are $\Ga$-invariant, $\frac{d \mu}{d T_{\varphi}^* \mu}$ is $\Ga$-invariant as well. Since $T_{\varphi}$ commutes with the $\Ga$-action, $T_{\varphi}^*\mu$ is $\Ga$-ergodic, and hence  $\frac{d \mu}{d T_{\varphi}^* \mu}$ is constant $T_{\varphi}^*\mu$-a.e., which must be positive. Hence, there exists $\la \in \R$ such that $\frac{d T_{\varphi}^* \mu}{d \mu} = e^{\la}$ $\mu$-a.e., and moreover, $\la \ge 0$ by Equation \eqref{eqn:trabscont}. This completes the proof.
\end{proof}

\subsection{Closed orbits in $\mathcal{H}$}

We record following observation that every $\Ga$-orbit outside $\La(\Ga) \times \R^r$ is closed. This implies that any $\Ga$-invariant ergodic Radon measure on $\mathcal{H} \smallsetminus (\La(\Ga) \times \R^r)$ is the counting measure of a single $\Ga$-orbit there, up to a constant multiple.

\begin{proposition} \label{prop:closedorbits}
    Let $\Ga < \Isom(Z)$ be a non-elementary transverse subgroup. Then for any $(\xi, u) \in \mathcal{H} \smallsetminus (\La(\Ga) \times \R^r)$,
    $$
    \Ga \cdot (\xi, u) \quad \text{is closed in } \mathcal{H}.
    $$
    
\end{proposition}

\begin{proof}
Suppose not. Then there exists a sequence $\{ g_n \}_{n \in \N} \subset \Ga$ such that $g_n (\xi, u) = (g_n \xi, u + \beta_{\xi}(g_n^{-1} z_0, z_0))$ converges in $\mathcal{H}$, to a point in $\mathcal{H} \smallsetminus \Ga \cdot~(\xi, u)$. In particular, the sequence $\{g_n\}_{n \in \N}$ is an infinite sequence. Hence, after passing to a subsequence, we can set $\zeta := \lim_{n \to + \infty} g_n^{-1} z_0 \in \La(\Ga)$. Since $\xi \notin \La(\Ga)$, at least one component of $\xi$ and $\zeta$ are different.  Therefore, $\beta_{\xi}(g_n^{-1} z_0, z_0)$ is unbounded, yielding a contradiction.
\end{proof}

\subsection{Proof of the rigidity}

Let us now prove Theorem \ref{thm:uniqueRadon}. 

\begin{proof}[Proof of Theorem \ref{thm:uniqueRadon}]

The case (2) is a direct consequence of Proposition \ref{prop:closedorbits}. We now prove (1). Let $\mu$ be a $\Ga$-invariant ergodic Radon measure on $\mathcal{H}$ supported on $\La_c(\Ga) \times \R^r$. We define 
\[
A := \left\{ a \in \mathbb{R}^r  : \exists \lambda (a) \in \R \text{ such that } \frac{d T_{a}^{*} \mu}{d \mu} = e^{\lambda(a)} \text{ a.e.} \right\}.
\]
It is straightforward that $A$ is an additive subgroup of $\mathbb{R}^r$ and $\lambda : A \rightarrow \R$ is an additive homomorphism. Moreover, by Theorem \ref{thm:trbytrlengthqi},
$$
\Spec(\Ga) \subset A.
$$
Hence, it follows from non-arithmeticity of $\Spec(\Ga)$ that $A \subset \R^r$ is dense.

\begin{claim*}
    The homomorphism $\la$ extends to a linear form $\la : \R^r \to \R$ so that    $$
    T_a^* \mu = e^{\la(a)} \cdot \mu. \quad( \forall a \in \R^{r})
    $$
\end{claim*}

To see the claim, let $f : \mathcal{H} \to \R$ be a compactly supported continuous function with $\int f d\mu > 0$. We define a map $\la_f : \R^r \to \R$ by
$$
e^{\la_f(a)} \int f \,d\mu = \int f \circ T_{-a} \,d \mu. \quad (\forall a \in \mathbb{R}^{r})
$$
Then $\la_f(a) = \la(a)$ for $a \in A$. By Dominated convergence theorem, $\la_f$ is continuous on $\R^{r}$. Since $\la : A \to \R$ is a homomorphism, this implies that $\la_f : \R^r \to \R$ is a continuous homomorphism, which must be a linear form. 

We apply the above argumet for every compactly supported continuous functions with positive integrals. Since the resulting linear form $\la_f$ conincides with $\la$ on a dense subset $A \subset \R^r$, $\la_f$ in fact does not depend on the choice of $f$, and is the unique extension of $\lambda : A \rightarrow \R$. 
That means, $e^{\lambda(a)} \int f \, d\mu =\int f \circ T_{-a} \,d \mu$ holds for \emph{every} $f \in C_{c}(\mathcal{H})$, where we mean by $\lambda$ the unique extension of $\lambda : A \rightarrow \R$.  This proves the claim.

\medskip

The claim implies that there exists a finite Borel measure $\nu_0$ on $\partial Z$ so that $\mu$ is decomposed on $\mathcal{H} = \partial Z \times \R^r$  as follows:
$$
d \mu(\xi, u) = e^{\la(u)} \cdot d \nu_0(\xi) \, du.
$$

By the $\Ga$-invariance of $\mu$, it is easy to see that for each $g \in \Ga$,
$$
\frac{d g_* \nu_0}{d \nu_0}(\xi) = e^{-\la(\beta_{\xi}( g z_0, z_0))} \quad \text{for $\nu_0$-a.e. $\xi \in \partial Z$}.
$$
Then for $z \in Z$, define the measure 
$\nu_{z}$ on $\partial Z$ by setting 
$$d \nu_z(\xi) :=  \frac{e^{- \la(\beta_{\xi}(z, z_0))}}{\nu_0(\partial Z)} d \nu_0 (\xi).$$
This is well-defined, and moreover the family $\{ \nu_z \}_{z \in Z}$ is a $1$-dimensional $\la$-conformal density of $\Ga$. Since $\{ \nu_z \}_{z \in Z}$ is supported on $\La_c(\Ga)$,  $\delta_{\la}(\Ga) = 1$ and $\Ga$ is of $\la$-divergence type by Theorem \ref{thm:BCZZHTS}. Therefore,
$$
\mu = \frac{1}{\nu_0(\partial Z)} \cdot \mu_{\la},
$$
which completes the proof.
\end{proof}

\section{Existence of ergodic invariant Radon measures} \label{sec:ergodicity}

We continue the setting of Section \ref{sec:UE}. In this section, we prove the ergodicity of the invariant Radon measure defined in Definition \ref{def:candidateergodicmeasure}. 

\begin{theorem} \label{thm:ergodicdiv}
    Let $\Ga < \Isom(Z)$ be a non-elementary transverse subgroup with non-arithmetic length spectrum. For a linear form $\psi : \R^r \to \R$, if  $\Ga$ is of $\psi$-divergence type, then 
    $$
    \text{the $\Ga$-action on $(\mathcal{H}, \mu_{\psi})$ is ergodic.}
    $$
    Moreover, $\mu_{\psi}$ is supported on $\La_{c}(\Ga) \times \R^r \subset \mathcal{H}$. 
\end{theorem}

Note that $\mu_{\psi}$ being supported on $\La_{c}(\Ga) \times \R^r$ is due to Blayac--Canary--Zhu--Zimmer \cite{BCZZ_PS} (Theorem \ref{thm:BCZZHTS}). 
Hence, it suffices to show that $\mu_{\Ga}$ is $\Ga$-ergodic. This is a special case of the following, together with Theorem~\ref{thm:BCZZHTS}:

\begin{theorem} \label{thm:ergodicnormal}
    Let $\Ga < \Isom(Z)$ be a non-elementary transverse subgroup and $\psi : \R^r \to \R$ a linear form. Suppose that $\Ga$ is of $\psi$-divergence type. Let $\Ga_0 \triangleleft \Ga$ be a normal subgroup such that
    \begin{itemize}
        \item $\Spec (\Ga_0)$ is non-arithmetic and
        \item the $\Ga_0$-action on $\partial Z$ is ergodic with respect to the $\delta_{\psi}(\Ga)$-dimensional $\psi$-conformal density of $\Ga$.
    \end{itemize}
    Then, 
    $$ 
    \text{the $\Ga_0$-action on $(\mathcal{H}, \mu_{\psi})$ is ergodic}
    $$
    where $\mu_{\psi}$ is the measure defined in Definition \ref{def:candidateergodicmeasure} for $\Ga$.
\end{theorem}

\subsection{Concentration on guided limit sets} \label{subsection:fullGuided}

We first strengthen the Hopf--Tsuji--Sullivan dichotomy of Blayac--Canary--Zhu--Zimmer \cite{BCZZ_PS} stated in Theorem \ref{thm:BCZZHTS}, by showing that the divergence-type conformal measure is in fact supported on guided limit sets. Recall from Definition \ref{def:divtypegpandmeasure} that a conformal density of $\Ga$ is of divergence type, if $\Ga$ is of divergence type with respect to a linear form associated to the given conformal density.

\begin{proposition}
    \label{prop:pattersonSqueeze}

    Let $\Ga < \Isom(Z)$ be a non-elementary transverse subgroup and $\nu = \{ \nu_z \}_{z \in Z}$ a divergence-type conformal density of $\Ga$. Let $\varphi \in \Ga$ be loxodromic  and let $C = C(\varphi) > 0$ be as in Lemma \ref{lem:BGIPHeredi}. Then
    $$
    \nu_{z_0} ( \La_{\varphi, C}(\Ga)) = 1.
    $$
\end{proposition}

\begin{proof}

    We consider the measure $\mu_{\nu}$ on $\mathcal{H}$ defined in Definition \ref{def:candidateergodicmeasure}. By Theorem \ref{thm:BCZZHTS}, we have that $\mu_{\nu}$ is supported on $\La_{c}(\Ga) \times \R^r$. Hence, we proceed the argument in the proof of Theorem \ref{thm:radonCharge} with $\mu_{\nu}$. Then for the subset $B_n \subset \mathcal{H}$ defined in Equation \eqref{eqn:defofBn}, $n \in \N$, we have $$\mu_{\nu}(B_n) > 0$$
    by Equation \eqref{eqn:positiveBn}. For each $n \in \N$, let $E_n \subset \partial Z$ be the projection of $B_n \subset \mathcal{H}$ to the $\partial Z$-component. Then by the definition of $\mu_{\nu}$, we have
    $$
    \nu_{z_0}(E_n) > 0 \quad \text{for all } n \in \N.
    $$
    
    In particular, $\Ga E_n \subset \partial Z$ is a $\Ga$-invariant subset of positive $\nu_{z_0}$-measure. This implies $\nu_{z_0}(\Ga E_n) = 1$ by the $\Ga$-ergodicity (Theorem \ref{thm:BCZZHTS}). Therefore, we have
    $$
    \nu_{z_0} \left( \bigcap_{n \in \N} \Ga E_n \right) = 1.
    $$
    Then as at the end of the proof of Theorem \ref{thm:radonCharge}, we have $$ \bigcap_{n \in \N} \Ga E_n \subset \La_{\varphi, C},$$ finishing the proof.
\end{proof}

\subsection{Essential subgroups}

An important ingredient to show the ergodicity of a measure on $\mathcal{H}$ is the notion of essential subgroups, introduced by Schmidt \cite{Schmidt1977cocycles} and studied further by Roblin \cite{Roblin2003ergodicite}.
For a conformal density $\nu = \{ \nu_z \}_{z \in Z}$, all measures in the family $\nu$ are in the same measure class. Hence, in discussing positivity of a Borel subset, we simply use the notation~$\nu$. 

\begin{definition} \label{def:ess}
    Let $\Ga < \Isom(Z)$ be a subgroup and let $\nu$ be a conformal density of $\Ga$. We define the subset $\ess_{\nu}(\Ga) \subset \R^r$ as follows: $a \in \ess_{\nu}(\Ga)$ if for each $\epsilon > 0$ and a Borel subset $E \subset \partial Z$ with $\nu(E) > 0$, there exists $g \in \Ga$ such that
    $$
\nu \left( E \cap g \varphi g^{-1} E \cap \{ \xi \in \partial Z: \norm{\beta_{\xi}(z_0, g \varphi g^{-1} z_0) - a}_{\infty} < \epsilon \}\right) > 0.
    $$
    It is easy to see that $\ess_{\nu}(\Ga)$ is a closed subgroup of $\R^r$. We call $\ess_{\nu}(\Ga)$ the \emph{essential subgroup} for $\Ga$ and $\nu$.
\end{definition}

This vector version of essential subgroup was introduced by Lee--Oh \cite{LO_invariant} for higher rank Lie groups.
The size of the essential subgroup plays a role of criterion for the ergodicity of actions on $\mathcal{H}$. The following was proved in \cite{Schmidt1977cocycles} for abstract measurable dynamical systems, and more direct proof for a particular case of $\op{CAT}(-1)$ spaces was given in \cite{Roblin2003ergodicite}. For a general higher rank Lie groups, this was proved by Lee--Oh \cite{LO_invariant}. The same proof works in our setting as well.

\begin{proposition}[{\cite{Schmidt1977cocycles}, \cite[Proposition 2.1]{Roblin2003ergodicite}, \cite[Proposition 9.2]{LO_invariant}}] \label{prop.essanderg}
    Let $\Ga < \Isom(Z)$ and let $\nu$ be a conformal density of $\Ga$. Then the $\Ga$-action on $(\Hor, \mu_{\nu})$ is ergodic if and only if the $\Ga$-action on $(\partial Z, \nu)$ is ergodic and $\ess_{\nu}(\Ga) = \R^r$.
\end{proposition}

In this perspective, the following is the main step in the proof of Theorem~\ref{thm:ergodicnormal}, which was proved by Roblin \cite{Roblin2003ergodicite} for $\CAT(-1)$ spaces. Roblin's approach was generalized to certain higher-rank settings by Lee--Oh \cite{LO_invariant} and by the second author \cite{kim2024conformal}  in different ways. While similar approaches would work for our setting as well, we present another proof that does not require metrizing the boundary for future applications.

\begin{lem}\label{lem:PSEssential}
    Let $\Ga < \Isom(Z)$ be a non-elementary transverse subgroup and $\nu$ a divergence-type conformal density of $\Ga$. Let $\varphi \in \Ga$ be loxodromic. Then for each $\epsilon > 0$ and a Borel subset $E \subset \partial Z$ with $\nu(E) > 0$, there exists $g \in \Ga$ such that $$
    \nu \left( E \cap g \varphi g^{-1} E \cap \{ \xi \in \partial Z: \norm{\beta_{\xi}(z_0, g \varphi g^{-1} z_0) - \tau_{\varphi}}_{\infty} < \epsilon \}\right) > 0.
    $$
    In particular, if $\Ga_0 \triangleleft \Ga$ is a normal subgroup, then
    $$
    \Spec (\Ga_0) \subset \ess_{\nu}(\Ga_0).
    $$
\end{lem}

\begin{proof}
    Let $C = C(\varphi_1) > 0$ be as in Lemma \ref{lem:BGIPHeredi}. By Proposition \ref{prop:pattersonSqueeze}, $\nu$ is supported on $\La_{\varphi, C}(\Ga)$.  Together with the inner regularity of $\nu$, it suffices to consider compact subsets of $\La_{\varphi, C}(\Ga)$.

    Denote by $\psi : \R^r \to \R$ a linear form associated to $\nu$. We can normalize $\psi$ so that $\delta_{\psi}(\Ga) = 1$, by Theorem \ref{thm:BCZZ_proper}.

    Let $K \subset \La_{\varphi, C}(\Ga)$ be a compact subset and fix $\epsilon > 0$. Suppose that for each $g \in \Ga$,
    $$
 \nu \left( K \cap g \varphi g^{-1} K \cap \{ \xi \in \partial Z : \norm{\beta_{\xi}(z_0, g \varphi g^{-1} z_0) - \tau_{\varphi}}_{\infty} < \epsilon \}\right)  = 0.
    $$
    Then showing $\nu(K) = 0$ finishes the proof.

    To do this, let $O \subset \partial Z$ be an open subset containing $K$. We will then construct a Borel subset $E(O) \subset O$ such that 
    \begin{equation} \label{eqn:fullessentialproof}
    \nu(K \cap E(O)) = 0 \quad \text{and} \quad \nu(E(O)) \ge  e^{-\psi(\tau_{\varphi}) - \epsilon \norm{\psi}_{\infty}} \cdot \nu(K).
    \end{equation}
    This yields $\nu(K) = 0$ as in \cite[Proof of Lemma 8.5]{CK_ML}.

Hence, it remains to find a set $E(O) \subset O$ satisfying Equation \eqref{eqn:fullessentialproof}. Recall the cover $\mathcal{U}$ and its subcollection  $\mathcal{V}$ for $K$ and $O$ constructed in the proof of Theorem \ref{thm:trbytrlengthqi}. For $l \in \N$, we also recall $C_l \subset O$ in Equation \eqref{eqn:defofCl} and the restriction $F_l = g_{i(l)} \varphi g_{i(l)}^{-1} : C_l \to \partial Z$ of the map in Equation \eqref{eqn:defofFl}, where $g_{i(l)} \in \Ga$ is given there. 

In the rest of this proof, we show that
$$
E(O) := \bigcup_{l \in \N} F_l(C_l \cap K)
$$
satisfies Equation \eqref{eqn:fullessentialproof}. By Equation \eqref{eqn:claimforFl}, we have  $\bigcup_{l \in \N} F_l(C_l \cap K) \subset \bigcup_{l \in \N} C_l \subset O$. In addition, by Equation \eqref{eqn:trqiclaim2goal2}, we have for each $l \in \N$ that 
\begin{equation} \label{eqn:FlClisinBuse}
F_l(C_l) \subset \left\{ \xi \in \partial Z : \norm{\beta_{\xi}(z_0, g_{i(l)} \varphi g_{i(l)}^{-1} z_0) - \tau_{\varphi}}_{\infty} < \epsilon \right\}.
\end{equation}
We then have
$$\begin{aligned}
 K & \cap F_l(C_l \cap K)  \\
& \subset K \cap g_{i(l)} \varphi g_{i(l)}^{-1} K \cap \left\{ \xi \in \partial Z : \norm{\beta_{\xi}(z_0, g_{i(l)} \varphi g_{i(l)}^{-1} z_0) - \tau_{\varphi}}_{\infty} < \epsilon \right\}.
\end{aligned}
$$
By our  hypothesis on $K$,
$
\nu( K \cap F_l (C_l \cap K) ) = 0.
$
Therefore, 
$$
\nu \left(  K \cap \bigcup_{l \in \N} F_l(C_l \cap K)  \right)= 0,
$$
showing the first claim in Equation \eqref{eqn:fullessentialproof}.

We now  estimate $\nu \left( \bigcup_{l \in \N} F_l(C_l \cap K) \right)$. By Equation \eqref{eqn:FlClisinBuse}, we have for each $l \in \N$ that
$$\begin{aligned}
\nu(F_l (C_l \cap K)) & = \int_{C_l \cap K} e^{-\psi( \beta_{\xi}( g_{i(l)} \varphi^{-1} g_{i(l)}^{-1} z_0, z_0))} d \nu(\xi) \\
& \ge  e^{-\psi(\tau_{\varphi}) - \epsilon \norm{\psi}_{\infty}} \nu(C_l \cap K).
\end{aligned}
$$
Since  $F_l(C_l \cap K)$'s are pairwise disjoint 
by Equation \eqref{eqn:claimforFl}, we have 
$$\begin{aligned}
\nu \left( \bigcup_{l \in \N} F_l(C_l \cap K) \right) 
& \ge  e^{-\psi(\tau_{\varphi}) - \epsilon \norm{\psi}_{\infty}} \sum_{l \in \N} \nu(C_l \cap K) \\
& \ge e^{-\psi(\tau_{\varphi}) - \epsilon \norm{\psi}_{\infty}} \cdot \nu \left( \bigcup_{l \in \N} (C_l \cap K) \right) 
\end{aligned}
$$
Since $K \subset \bigcup_{l \in \N} C_l$ as in Equation \eqref{eqn:defofCl}, this implies the second claim in Equation \eqref{eqn:fullessentialproof}. 
\end{proof}

\begin{proof}[Proof of Theorem \ref{thm:ergodicnormal}]
    Now Theorem \ref{thm:ergodicnormal} is a consequence of Proposition \ref{prop.essanderg} and Lemma \ref{lem:PSEssential}.
\end{proof}

\section{Higher-rank homogeneous spaces} \label{section:homogeneous}

In this section, we deduce Theorem \ref{thm:mainAnosov}, Corollary \ref{cor:mainAnosov}, Theorem \ref{thm:maintransverse}, and Corollary \ref{cor:mainrelAnosov}.
Let $G$ be a connected semisimple real algebraic group.
 Recall from the introduction that $P < G$ is a minimal parabolic subgroup with a Langlands decomposition $P = MAN$, where $A$ is a maximal real split torus, $M < P$ is a maximal compact subgroup commuting with $A$, and $N$ is the unipotent radical of $P$. 
 We also chose a maximal compact subgroup $K < G$ so that we have the Cartan decomposition $G = K(\exp \fa^+) K$, where $\fa^+ \subset \Lie A =: \fa$ is a fixed positive Weyl chamber. Denote the Cartan projection by $\kappa : G \to \fa^+$, defined by the condition that $g \in K (\exp \kappa(g)) K$ for all $g \in G$.

 We have the Iwasawa decomposition $G = KAN$ and the Furstenberg boundary is $\mathcal{F} = G/P = K/M$. For $\xi \in \mathcal{F}$ and $g \in G$, the Iwasawa cocycle $\sigma(g, \xi) \in \fa$ is the element such that $gk \in K (\exp \sigma(g, \xi)) N$ where $k \in K$ is such that $\xi = k M \in \mathcal{F}$.
 Then the \emph{$\fa$-valued Busemann cocycle} $\beta : \mathcal{F} \times G \times G \to \fa$ is defined as follows: for $\xi \in \mathcal{F}$ and $g, h \in G$,
 \begin{equation} \label{eqn:Busedef}
    \beta_{\xi}(g, h) := \sigma(g^{-1}, \xi) - \sigma(h^{-1}, \xi).
 \end{equation}

 \begin{example}
    We present a specific example $G = \PSL(2, \mathbb{C}) = \Isom^+(\H^3)$, regarded as a real algebraic grooup $\so(3,1)$. In this case, we can choose subgroups as follows:
    $$\begin{aligned} 
    P & := \left\{ \begin{pmatrix}
    a & b \\
    0 & 1/a
    \end{pmatrix} : a, b \in \mathbb{C}, \ a \neq 0 \right\} \\
  M & := \left\{ \begin{pmatrix}
    e^{i \theta /2} & 0 \\
    0 & e^{-i \theta /2}
  \end{pmatrix} : \theta \in \R \right\} \simeq \op{PSU}(1) \simeq \S^1 \\
  A & : = \left\lbrace \begin{pmatrix}
    e^{t/2} & 0 \\
    0 & e^{-t/2}
  \end{pmatrix} : t \in \R \right\rbrace \simeq \R \\
  N & := \left\lbrace \begin{pmatrix}
    1 & z \\
    0 & 1
  \end{pmatrix} : z \in \mathbb{C} \right\rbrace \\
    K & := \left\{ \begin{pmatrix}
    a & b \\
    -\ov{b} & \ov{a}
    \end{pmatrix} : a, b \in \mathbb{C}, \ |a|^2 + |b|^2 = 1 \right\} \simeq \op{PSU}(2)\\
\end{aligned}
$$
Using the upper half-space model of $\H^3$, its boundary is the Riemann sphere $\widehat{\mathbb{C}} = \mathbb{C} \cup \{\infty\}$ on which $G$ acts as linear fractional transformations. Then $P = \stab_{G}(\infty)$, and the Furstenberg boundary $\mathcal{F} = G/P$ is the same as the Riemann sphere $\widehat{\mathbb{C}}$. Busemann cocycles are defined as usual.
 \end{example}

\subsection{As a product of $\CAT(-1)$ spaces}

In the rest of this section,  we now consider the case as in Equation \eqref{eqn:standing} that 
$$
G :=  \prod_{i = 1}^r G_i
$$
where $G_i$ is a simple real algebraic group of rank one. 

For each $1 \le i \le r$, we fix corresponding objects $P_i$, $M_i$, $A_i$, $N_i$, $K_i$, $\fa_i^+$, and $\fa_i$ for $G_i$. Then we can make the choices for $G$ by setting 
$$
\heartsuit = \prod_{i = 1}^r \heartsuit_i 
$$
for each $\heartsuit \in \{ P, M, A, N, K, \fa^+, \fa\}$.

For each $1 \le i \le r$, we denote the Riemannian symmetric space associated to $G_i$ by 
$$
X_i := G_i/K_i,
$$
and equip it with the left $G_i$-invariant and right $K_i$-invariant metric induced by the Killing form on $\fa_i$. Then $X_i$ is a proper geodesic $\CAT(-1)$ space, with the Gromov boundary
$$
\partial X_i = K_i / M_i = G_i / P_i.
$$
Hence, we have 
$$
G/K = \prod_{i = 1}^r X_i \quad \text{and} \quad \mathcal{F} = \prod_{i = 1}^r \partial X_i
$$
which enable us to use results in Section \ref{sec:product}, Section \ref{sec:UE}, and Section \ref{sec:ergodicity}.

Indeed, fixing a basepoint $z_0 = [\id] \in G/K$, we have
$$
\kappa(g) = \kappa(z_0, g z_0) \quad \text{for all } g \in G
$$
where $\kappa(\cdot, \cdot)$ is defined as in Equation \eqref{eqn:proddistance} for $Z = G/K$. In addition, we have
$$
\beta_{\xi}(\id, g) = \beta_{\xi}(z_0, g z_0) \quad \text{for all } g \in G, \xi \in \mathcal{F}
$$
where $\beta$ on the right hand side is defined as in Equation \eqref{eqn:prodBuse}.

Employing the notions introduced in Section \ref{sec:UE}, for a loxodromic $g \in G$, its vector-valued translation length
$$
\tau_{g} = \lim_{n \to + \infty} \frac{\kappa(g^n)}{n} \in \fa^+
$$
is also called the Jordan projection of $g \in G$.

\subsection{Discrete subgroups} We mainly consider a discrete subgroup $\Ga < G$. Recall from Definition \ref{def:spectrum} the length spectrum of $\Ga$
$$
\Spec(\Ga) = \{ \tau_g \in \fa : g \in \Ga, \text{ loxodromic}\}
$$
and that $\Spec(\Ga)$ is called non-arithmetic if it generates a dense additive subgroup of $\fa$. As shown by Benoist, Zariski density gives non-arithmeticity of length spectrum.

\begin{theorem}[{\cite{Benoist2000proprietes}}] \label{thm:nonarithmetic}
    Let $\Ga < G$ be a Zariski dense discrete subgroup. Then $\Spec(\Ga)$ is non-arithmetic.
\end{theorem}

The limit set is defined as in Definition \ref{def:limitproduct}. Similarly, the notion of \emph{transverse subgroup} of $G$ is defined as in Definition \ref{def:transverseintro} or Definition \ref{def:transprod}. Conical limit set is defined as in Definition \ref{dfn:conical}.

In the introduction, Anosov subgroups and relatively Anosov subgroups are defined as transverse subgroups that act on their limit sets as uniform convergence groups and geometrically finite convergence groups, respectively. We present slightly different but equivalent definitions here. These formulations are motivated by the study of Gromov \cite{gromov1987hyperbolic}, Bowditch \cite{bowditch1998a-topological}, and Yaman \cite{Yaman2004topological} regarding hyperbolic and relatively hyperbolic groups in terms of convergence actions. As for loxodromic elements, we call $g \in G$ \emph{parabolic} if each component of $g$ is parabolic. 

\begin{definition} \label{def:Anosovdefbody}
    Let $\Ga < G$ be a non-elementary transverse subgroup.
    \begin{itemize}
        \item We call $\Ga$ \emph{Anosov} if $\Ga$ is a hyperbolic group and there exists a $\Ga$-equivariant homeomorphism $\partial \Ga \to \La(\Ga)$, where $\partial \Ga$ is the Gromov boundary of $\Ga$.
        \begin{itemize}
            \item Equivalently, $\La(\Ga) = \La_{c}(\Ga)$.
        \end{itemize}
        \item We call $\Ga$ \emph{relatively Anosov} if $\Ga$ is a relatively hyperbolic group (with some choice of a peripheral structure) and there exists a $\Ga$-equivariant homeomorphism $\partial_B \Ga \to \La(\Ga)$, where $\partial_B \Ga$ is the Bowditch boundary of $\Ga$ with respect to the chosen peripheral structure.
        \begin{itemize}
            \item Equivalently, $\La(\Ga) = \La_{c}(\Ga) \sqcup \La_{p}(\Ga)$, where $\La_{p}(\Ga)$ is the parabolic limit set of $\Ga$, i.e., the set of all fixed points of parabolic elements of $\Ga$.
        \end{itemize}
    \end{itemize}
\end{definition}

Another equivalent characterization of Anosov and relatively Anosov subgroups are as follows: a subgroup $\Ga < G$ is Anosov if and only if there exist a non-elementary convex cocompact subgroup $\widehat{\Ga}_1 < G_1$ and a faithful convex cocompact representation $\rho_i : \widehat{\Ga}_1 \to G_i$ for each $2 \le i \le r$ so that the diagonal embedding $(\id \times \rho_2 \times \cdots \times \rho_r)(\widehat{\Ga}_1) < G$ is a finite index subgroup of $\Ga$. Similarly, $\Ga$ is relatively Anosov if and only if there exist a non-elementary geometrically finite subgroup $\widehat{\Ga}_1 < G_1$ and a type-preserving geometrically finite representation $\rho_i : \widehat{\Ga}_1 \to G_i$ for each $2 \le i \le r$ so that the diagonal embedding $(\id \times \rho_2 \times \cdots \times \rho_r)(\widehat{\Ga}_1) < G$ is a finite index subgroup of $\Ga$.

\begin{remark}
    Using \cite[Proposition 5.7]{KO_Entropy}, it is easy to see that Theorem \ref{thm:radonCharge} and Proposition \ref{prop:pattersonSqueeze} hold for relatively Anosov subgroups of a general semisimple real algebraic group, where the alignment is discussed in Gromov models for relatively hyperbolic groups. Similarly, they also hold for the class of hypertransverse subgroups in the sense of \cite{kim2024conformal}, which is the same as the class of transverse subgroups when the ambient group is a product of rank-one Lie groups.
\end{remark}

Recall the Burger--Roblin measure $\mu_{\nu}^{\BR}$ associated to a conformal measure $\nu$, from Equation \eqref{eqn:BRdef}. Its ergodicity was shown as follows:

\begin{theorem}[{\cite{LO_invariant}, \cite{LO_ergodic}, \cite{kim2024conformal}}] \label{thm:BRergodicknown}
    Let $\Ga < G$ be a Zariski dense transverse subgroup. For a divergence-type conformal measure $\nu$ of $\Ga$, the Burger--Roblin measure $\mu_{\nu}^{\BR}$ is $N$-ergodic.
\end{theorem}

This ergodicity was proved for Anosov subgroups in \cite{LO_ergodic}, in which case the $NM$-ergodicity was shown in \cite{LO_invariant}. The $N$-ergodicity for transverse subgroups was proved in \cite{kim2024conformal}.

\subsection{Measure classifications}
We now complete the deduction of our measure classification results. Setting $\mathcal{H} := \mathcal{F} \times \fa$, define the map
$$
\begin{aligned}
G & \quad \to \quad \mathcal{H}  = \mathcal{F} \times \fa\\
g & \quad \mapsto \quad (g P, \ \beta_{g P}(\id , g))
\end{aligned}
$$
which induces the homeomorphism
$$
G/NM \to \mathcal{H}.
$$
Via this homeomorphism, the left multiplication of $G$ on $G/NM$ descends to the $G$-action on $\mathcal{H}$ defined as follows: for $g \in G$ and $(\xi, u) \in \mathcal{H}$,
$$
g \cdot (\xi, u) = (g \xi, u + \beta_{\xi}(g^{-1}, \id )).
$$
Therefore, $\mathcal{H}$ is indeed the same as the horospherical foliation of the product $G/K = \prod_{i = 1}^r X_i$ of $\CAT(-1)$ spaces defined as in Equation \eqref{eqn:horofoliation}.

Then for a subgroup $\Ga < G$, any $NM$-invariant Radon measure on $\Ga \ba G$ is induced by a $\Ga$-invariant measure on $G$ of the form
$$
d \widehat \mu(\xi, u)\, dn dm 
$$
for some $\Ga$-invariant Radon measure $\widehat \mu$ on $\mathcal{H}$, where $dn$ and $dm$ are Haar measures on $N$ and $M$ respectively (cf. \cite[Proposition 10.25]{LO_invariant}). Hence, it suffices to classify $\Ga$-invariant Radon measures on $\mathcal{H}$.

\medskip 

We first deduce Theorem \ref{thm:maintransverse}.  Let $\Ga < G$ be a Zariski dense transverse subgroup.  The $NM$-ergodicity and $N$-ergodicity of Burger--Roblin measures (Equation \eqref{eqn:BRdef}) associated to divegence-type conformal measures on $\La(\Ga)$ were proved in \cite{kim2024conformal}.  In other words, we have inclusions $(1) \subset (2)$ and $(1) \subset (3)$ in the statement. Hence, it remains to show that those are all ergoic measures.

The recurrence locus $\mathcal{R}_{\Ga} \subset \Ga \ba G$ in Equation \eqref{eqn:recurrencelocus} is characterized as
$$
\mathcal{R}_{\Ga} = \{ [g] \in \Ga \ba G : g P \in \La_{c}(\Ga) \}.
$$
Hence, classifying $NM$-invariant ergodic Radon measures supported on $\mathcal{R}_{\Ga}$ is equivalent to classifying $\Ga$-invariant ergodic Radon measures supported on $\La_{c}(\Ga) \times \fa \subset \mathcal{H}$. Together with the non-arithmeticity (Theorem \ref{thm:nonarithmetic}), it follows from Theorem \ref{thm:uniqueRadon} that any $NM$-invariant ergodic measure on $\mathcal{R}_{\Ga}$ is the Burger--Roblin meaure associated to a divergence-type conformal measure of $\Ga$ on $\La(\Ga)$, up to a constant multiple. This shows the equality $(1) = (2)$ in the statement.

Combining the classification that all $NM$-invariant ergodic Radon measures on $\mathcal{R}_{\Ga}$ are Burger--Roblin measures and the $N$-ergodicity of Burger--Roblin measures, it follows that all $N$-invariant ergodic Radon measures on $\mathcal{R}_{\Ga}$ are Burger--Roblin measures, as in \cite[Corollary 6.5]{Winter_mixing}. This finishes the proof, showing $(1) = (3)$ in the statement.

\medskip

We now deduce Corollary \ref{cor:mainrelAnosov}. Let $\Ga < G$ be a Zariski dense relatively Anosov subgroup. Note that for any $(\xi, u) \in \mathcal{H}$, either $\xi \in \La_{c}(\Ga)$, $\xi \in \La_{p}(\Ga)$, or $\xi \notin \La(\Ga)$. In last two cases, the orbit $\Ga \cdot (\xi, u)$ is closed in $\mathcal{H}$ by the characterization of relatively Anosov subgroups and Proposition \ref{prop:closedorbits}. Therefore, Corollary \ref{cor:mainrelAnosov} follows from Theorem \ref{thm:maintransverse}.

\medskip

Finally, let $\Ga < G$ be a Zariski dense Anosov subgroup. Then $\La(\Ga) = \La_{c}(\Ga)$, and hence $\E_{\Ga} = \mathcal{R}_{\Ga}$. Therefore, Theorem \ref{thm:mainAnosov} and Corollary \ref{cor:mainAnosov} are special cases of Corollary \ref{cor:mainrelAnosov} and Theorem \ref{thm:maintransverse} respectively.

\appendix
\section{Some hyperbolic geometry} \label{appendix}

We first prove Lemma \ref{lem:CAT(-1)Fellow}.

\begin{proof}[Proof of Lemma \ref{lem:CAT(-1)Fellow}]
Consider the $\pi/2-0-0$ triangle in $\mathbb{H}^{2}$ with a vertex $O$ and two ideal vertices $\xi, \zeta \in \partial \mathbb{H}^{2}$. Then $d(O, \overline{\xi\zeta}) = 2\tanh^{-1}(1-1/\sqrt{2}) \approx 0.60346$.

Let $z \in [x, y]$ be the earliest point such that $d(\pi_{\gamma}(x), \pi_{\gamma}(z)) =1$. We will then take $p \in [x, z]$ such that $d(\pi_{\gamma}(x), p) \le 2$. 

Let $P \in [x, \pi_{\gamma}(z)]$ be the nearest point from $\pi_{\gamma}(x)$. By comparison with the hyperbolic triangle, $d(P, \pi_{\gamma}(x)) < 0.604$. Hence, we have \begin{equation}\label{eqn:0396}
d(p, \pi_{\gamma}(y)) \ge 1 - 0.604 \ge 0.396. 
\end{equation}

 Now, since $d(\pi_{\gamma}(x), \pi_{\gamma}(z)) > 1$, the angle $\measuredangle \pi_{\gamma}(z) \pi_{\gamma}(x) P$ is less than 45 degrees. Now, let $Q \in [x, \pi_{\gamma}(z)]$ be the nearest point from $z$. Note that $\measuredangle y \pi_{\gamma}(z) \pi_{\gamma}(x)$ is 90 degrees and $\measuredangle \pi_{\gamma}(z) \pi_{\gamma}(x) x$ is at most 45 degrees. Hence, $\measuredangle z \pi_{\gamma}(z) x$ is at least 45 degrees. This implies that $d(\pi_{\gamma}(z), Q)$ is at most $0.604$. This implies that $P$ either comes earlier than $q$ on $[x, \pi_{\gamma}(z)]$, or comes no later than $Q$ by $0.208$.
 
 Now note that $\triangle z Q x$ is a right-angled triangle. The comparison principle tells us that $d(Q, [x, z]) < 0.604$. If $P$ comes earlier than $Q$ along $[x, \pi_{\gamma}(z)]$, then Lemma \ref{fact:CAT(-1)Thin} says $d(P, [x, z]) < 0.604$ as well. If $P$ comes later than $Q$ and hence $d(P, Q) < 0.208$, then we have $d(P, [x, z]) < 0.604+0.208 = 0.812$. Either way, we have $d(\pi_{\gamma}(x), [x, z]) < 0.604 + 0.812 \le 2$.

For the same reason, we can take the latest $z' \in [x, y]$ satisfying that $d(\pi_{\gamma}(z'), \pi_{\gamma}(y)) = 1$ and then take $q \in [z', y]$ such that $d(\pi_{\gamma}(y), q) \le 2$. These $p$ and $q$ work.\end{proof}

We next sketch the proof of Lemma \ref{lem:alignExtCts}. When $z \in X$, this is due to the 1-Lipschitzness of $\pi_{\gamma}(\cdot)$. Hence, suppose that $z \in \partial X$. Since $z_{n} \rightarrow z \in \partial X$, we have $d([z_{n}, z_{m}], \gamma) \rightarrow \infty$ as $n, m \rightarrow +\infty$.

Hence, it suffices to prove that: \begin{claim*}
for each $z, w \in X$, if $d(\pi_{\gamma}(z), \pi_{\gamma}(w)) = \epsilon>0$, then $d([z, w], \pi_{\gamma}(w)) \le 1+e^{2/l}$ where $l = \frac{e^{\epsilon}-1}{e^{\epsilon} + 1}$.
\end{claim*}

To see this, let $\triangle ABC$ be the comparison triangle in $\mathbb{H}^{2}$ for $\triangle z \pi_{\gamma}(z) \pi_{\gamma}(w)$ in $X$. Then $\measuredangle ABC \ge \measuredangle z \pi_{\gamma}(z) \pi_{\gamma}(w) = \pi/2$. This forces that $\measuredangle ACB$ is smaller than the angle $\measuredangle \xi C B$, where $\xi$ is the boundary point made by the ray $\overrightarrow{BA}$. By hyperbolic geometry, we have that \[
\measuredangle ACB \le \measuredangle \xi C B \le \tan^{-1} \frac{1-l^{2}}{2l}. 
\]
We then have $\measuredangle w \pi_{\gamma}(w) z \ge \pi/2 - \measuredangle ACB \ge \tan^{-1} \frac{2l}{1-l^{2}}$.

Let us draw a comparison triangle $\triangle PQR$ in $\mathbb{H}^{2}$ for $\triangle w \pi_{\gamma}(w) z$ in $X$. Then $\measuredangle PQR \ge \tan^{-1} \frac{2l}{1-l^{2}}$, and there exist $p \in \overline{PQ}$ and $q \in \overline{QR}$ with $d(p, q) \le 1$ and $d(p, Q) \le e^{2/l}$. This implies that $d(\pi_{\gamma}(w), [z, w]) \le 1 + e^{2/l}$.

\bibliographystyle{alpha} 
\bibliography{ML}

\begin{thebibliography}{BCZZ24b}

\bibitem[ANSS02]{aaronson2002invariant}
Jon Aaronson, Hitoshi Nakada, Omri Sarig, and Rita Solomyak.
\newblock Invariant measures and asymptotics for some skew products.
\newblock {\em Israel J. Math.}, 128:93--134, 2002.

\bibitem[Bab04]{Babillot_nilpotent}
Martine Babillot.
\newblock On the classification of invariant measures for horosphere foliations
  on nilpotent covers of negatively curved manifolds.
\newblock In {\em Random walks and geometry}, pages 319--335. Walter de
  Gruyter, Berlin, 2004.

\bibitem[BCZZ24a]{BCZZ_2}
Pierre-Louis Blayac, Richard Canary, Feng Zhu, and Andrew Zimmer.
\newblock Counting, mixing and equidistribution for {GPS} systems with
  applications to relatively anosov groups.
\newblock {\em arXiv preprint arXiv:2404.09718}, 2024.

\bibitem[BCZZ24b]{BCZZ_PS}
Pierre-Louis Blayac, Richard Canary, Feng Zhu, and Andrew Zimmer.
\newblock Patterson-{S}ullivan theory for coarse cocycles.
\newblock {\em arXiv preprint arXiv:2404.09713}, 2024.

\bibitem[Ben97]{Benoist1997proprietes}
Y.~Benoist.
\newblock Propri\'{e}t\'{e}s asymptotiques des groupes lin\'{e}aires.
\newblock {\em Geom. Funct. Anal.}, 7(1):1--47, 1997.

\bibitem[Ben00]{Benoist2000proprietes}
Yves Benoist.
\newblock Propri\'et\'es asymptotiques des groupes lin\'eaires. {II}.
\newblock In {\em Analysis on homogeneous spaces and representation theory of
  {L}ie groups, {O}kayama--{K}yoto (1997)}, volume~26 of {\em Adv. Stud. Pure
  Math.}, pages 33--48. Math. Soc. Japan, Tokyo, 2000.

\bibitem[BH99]{Bridson1999metric}
Martin~R. Bridson and Andr\'e Haefliger.
\newblock {\em Metric spaces of non-positive curvature}, volume 319 of {\em
  Grundlehren der mathematischen Wissenschaften [Fundamental Principles of
  Mathematical Sciences]}.
\newblock Springer-Verlag, Berlin, 1999.

\bibitem[BL98]{BL_covers}
Martine Babillot and Fran\c{c}ois Ledrappier.
\newblock Geodesic paths and horocycle flow on abelian covers.
\newblock In {\em Lie groups and ergodic theory ({M}umbai, 1996)}, volume~14 of
  {\em Tata Inst. Fund. Res. Stud. Math.}, pages 1--32. Tata Inst. Fund. Res.,
  Bombay, 1998.

\bibitem[BLLO23]{BLLO}
Marc Burger, Or~Landesberg, Minju Lee, and Hee Oh.
\newblock The {H}opf-{T}suji-{S}ullivan dichotomy in higher rank and
  applications to {A}nosov subgroups.
\newblock {\em J. Mod. Dyn.}, 19:301--330, 2023.

\bibitem[Bow98]{bowditch1998a-topological}
Brian~H. Bowditch.
\newblock A topological characterisation of hyperbolic groups.
\newblock {\em J. Amer. Math. Soc.}, 11(3):643--667, 1998.

\bibitem[Bow08]{bowen2008equilibrium}
Rufus Bowen.
\newblock {\em Equilibrium states and the ergodic theory of {A}nosov
  diffeomorphisms}, volume 470 of {\em Lecture Notes in Mathematics}.
\newblock Springer-Verlag, Berlin, revised edition, 2008.
\newblock With a preface by David Ruelle, Edited by Jean-Ren\'{e} Chazottes.

\bibitem[Bur90]{Burger_horoc}
Marc Burger.
\newblock Horocycle flow on geometrically finite surfaces.
\newblock {\em Duke Math. J.}, 61(3):779--803, 1990.

\bibitem[CDP90]{coornaert1990geometrie}
M.~Coornaert, T.~Delzant, and A.~Papadopoulos.
\newblock {\em G\'{e}om\'{e}trie et th\'{e}orie des groupes}, volume 1441 of
  {\em Lecture Notes in Mathematics}.
\newblock Springer-Verlag, Berlin, 1990.
\newblock Les groupes hyperboliques de Gromov. [Gromov hyperbolic groups], With
  an English summary.

\bibitem[CK25]{CK_ML}
Inhyeok Choi and Dongryul~M. Kim.
\newblock Invariant measures on the space of measured laminations for subgroups
  of mapping class group.
\newblock {\em arXiv preprint arXiv:2510.23256}, 2025.

\bibitem[CK26]{CK_endgame}
Inhyeok Choi and Dongryul~M. Kim.
\newblock Classification of horospherical invariant measures in higher rank:
  {T}he {F}ull {S}tory.
\newblock {\em arXiv preprint arXiv:2601.22668}, 2026.

\bibitem[Cou24]{coulon2024patterson-sullivan}
R\'{e}mi Coulon.
\newblock Patterson-{S}ullivan theory for groups with a strongly contracting
  element.
\newblock {\em Ergodic Theory Dynam. Systems}, 44(11):3216--3271, 2024.

\bibitem[CZZ24]{CZZ_transverse}
Richard Canary, Tengren Zhang, and Andrew Zimmer.
\newblock Patterson-{S}ullivan measures for transverse subgroups.
\newblock {\em J. Mod. Dyn.}, 20:319--377, 2024.

\bibitem[Dan78]{dani1978invariant}
S.~G. Dani.
\newblock Invariant measures of horospherical flows on noncompact homogeneous
  spaces.
\newblock {\em Invent. Math.}, 47(2):101--138, 1978.

\bibitem[Dan81]{dani1981invariant}
S.~G. Dani.
\newblock Invariant measures and minimal sets of horospherical flows.
\newblock {\em Invent. Math.}, 64(2):357--385, 1981.

\bibitem[ELO23]{Edwards2020anosov}
Sam Edwards, Minju Lee, and Hee Oh.
\newblock Anosov groups: local mixing, counting and equidistribution.
\newblock {\em Geom. Topol.}, 27(2):513--573, 2023.

\bibitem[Fur73]{furstenberg1973the-unique}
Harry Furstenberg.
\newblock The unique ergodicity of the horocycle flow.
\newblock In {\em Recent advances in topological dynamics ({P}roc. {C}onf.
  {T}opological {D}ynamics, {Y}ale {U}niv., {N}ew {H}aven, {C}onn., 1972; in
  honor of {G}ustav {A}rnold {H}edlund)}, Lecture Notes in Math., Vol. 318,
  pages 95--115. Springer, Berlin-New York, 1973.

\bibitem[GdlH90]{ghys1990bord}
\'{E}tienne Ghys and Pierre de~la Harpe.
\newblock Le bord d'un espace hyperbolique.
\newblock In {\em Sur les groupes hyperboliques d'apr{\`e}s {M}ikhael {G}romov
  ({B}ern, 1988)}, volume~83 of {\em Progr. Math.}, pages 117--134.
  Birkh\"{a}user Boston, Boston, MA, 1990.

\bibitem[Gro87]{gromov1987hyperbolic}
M.~Gromov.
\newblock Hyperbolic groups.
\newblock In {\em Essays in group theory}, volume~8 of {\em Math. Sci. Res.
  Inst. Publ.}, pages 75--263. Springer, New York, 1987.

\bibitem[GW12]{Guichard2012anosov}
Olivier Guichard and Anna Wienhard.
\newblock Anosov representations: domains of discontinuity and applications.
\newblock {\em Invent. Math.}, 190(2):357--438, 2012.

\bibitem[Kim24]{kim2024conformal}
Dongryul~M. Kim.
\newblock Conformal measure rigidity and ergodicity of horospherical
  foliations.
\newblock {\em arXiv preprint arXiv:2404.13727}, 2024.

\bibitem[KLP17]{KLP_Anosov}
Michael Kapovich, Bernhard Leeb, and Joan Porti.
\newblock Anosov subgroups: dynamical and geometric characterizations.
\newblock {\em Eur. J. Math.}, 3(4):808--898, 2017.

\bibitem[KO25]{KO_Entropy}
Dongryul~M. Kim and Hee Oh.
\newblock Relatively {A}nosov groups: finiteness, measure of maximal entropy,
  and reparameterization.
\newblock {\em J. Reine Angew. Math.}, 826:91--142, 2025.

\bibitem[KOW25]{KOW_PD}
Dongryul~M. Kim, Hee Oh, and Yahui Wang.
\newblock Properly discontinuous actions, growth indicators, and conformal
  measures for transverse subgroups.
\newblock {\em Math. Ann.}, 393(2):2391--2450, 2025.

\bibitem[Lab06]{Labourie2006anosov}
Fran\c{c}ois Labourie.
\newblock Anosov flows, surface groups and curves in projective space.
\newblock {\em Invent. Math.}, 165(1):51--114, 2006.

\bibitem[Lan21]{Landesberg_horospherically}
Or~Landesberg.
\newblock Horospherically invariant measures and finitely generated {K}leinian
  groups.
\newblock {\em J. Mod. Dyn.}, 17:337--352, 2021.

\bibitem[Led08]{Ledrappier_invariant}
Fran\c{c}ois Ledrappier.
\newblock Invariant measures for the stable foliation on negatively curved
  periodic manifolds.
\newblock {\em Ann. Inst. Fourier (Grenoble)}, 58(1):85--105, 2008.

\bibitem[LL22]{LL_Radon}
Or~Landesberg and Elon Lindenstrauss.
\newblock On {R}adon measures invariant under horospherical flows on
  geometrically infinite quotients.
\newblock {\em Int. Math. Res. Not. IMRN}, (15):11602--11641, 2022.

\bibitem[LLLO23]{LLLO_Horospherical}
Or~Landesberg, Minju Lee, Elon Lindenstrauss, and Hee Oh.
\newblock Horospherical invariant measures and a rank dichotomy for {A}nosov
  groups.
\newblock {\em J. Mod. Dyn.}, 19:331--362, 2023.

\bibitem[LO23]{LO_invariant}
Minju Lee and Hee Oh.
\newblock Invariant measures for horospherical actions and {A}nosov groups.
\newblock {\em Int. Math. Res. Not. IMRN}, (19):16226--16295, 2023.

\bibitem[LO24]{LO_ergodic}
Minju Lee and Hee Oh.
\newblock Ergodic decompositions of geometric measures on {A}nosov homogeneous
  spaces.
\newblock {\em Israel J. Math.}, 260(1):195--234, 2024.

\bibitem[LS07]{LS_periodic}
Fran\c{c}ois Ledrappier and Omri Sarig.
\newblock Invariant measures for the horocycle flow on periodic hyperbolic
  surfaces.
\newblock {\em Israel J. Math.}, 160:281--315, 2007.

\bibitem[MT18]{maher2018random}
Joseph Maher and Giulio Tiozzo.
\newblock Random walks on weakly hyperbolic groups.
\newblock {\em J. Reine Angew. Math.}, 742:187--239, 2018.

\bibitem[Oh25]{oh2025dynamics}
Hee Oh.
\newblock Dynamics and {R}igidity through the {L}ens of {C}ircles.
\newblock {\em arXiv preprint arXiv:2510.10771}, 2025.

\bibitem[OP19]{OP_local}
Hee Oh and Wenyu Pan.
\newblock Local mixing and invariant measures for horospherical subgroups on
  abelian covers.
\newblock {\em Int. Math. Res. Not. IMRN}, (19):6036--6088, 2019.

\bibitem[Qui02a]{Quint2002divergence}
J.-F. Quint.
\newblock Divergence exponentielle des sous-groupes discrets en rang
  sup\'{e}rieur.
\newblock {\em Comment. Math. Helv.}, 77(3):563--608, 2002.

\bibitem[Qui02b]{Quint2002Mesures}
J.-F. Quint.
\newblock Mesures de {P}atterson-{S}ullivan en rang sup\'{e}rieur.
\newblock {\em Geom. Funct. Anal.}, 12(4):776--809, 2002.

\bibitem[Rat91]{Ratner_measure}
Marina Ratner.
\newblock On {R}aghunathan's measure conjecture.
\newblock {\em Ann. of Math. (2)}, 134(3):545--607, 1991.

\bibitem[Rob03]{Roblin2003ergodicite}
Thomas Roblin.
\newblock Ergodicit\'e{} et \'equidistribution en courbure n\'egative.
\newblock {\em M\'em. Soc. Math. Fr. (N.S.)}, (95):vi+96, 2003.

\bibitem[Sar04]{Sarig_abelian}
Omri Sarig.
\newblock Invariant {R}adon measures for horocycle flows on abelian covers.
\newblock {\em Invent. Math.}, 157(3):519--551, 2004.

\bibitem[Sar10]{Sarig_genus}
Omri Sarig.
\newblock The horocyclic flow and the {L}aplacian on hyperbolic surfaces of
  infinite genus.
\newblock {\em Geom. Funct. Anal.}, 19(6):1757--1812, 2010.

\bibitem[Sch77]{Schmidt1977cocycles}
Klaus Schmidt.
\newblock {\em Cocycles on ergodic transformation groups}, volume Vol. 1 of
  {\em Macmillan Lectures in Mathematics}.
\newblock Macmillan Co. of India, Ltd., Delhi, 1977.

\bibitem[Vee77]{veech1977unique}
William~A. Veech.
\newblock Unique ergodicity of horospherical flows.
\newblock {\em Amer. J. Math.}, 99(4):827--859, 1977.

\bibitem[Win15]{Winter_mixing}
Dale Winter.
\newblock Mixing of frame flow for rank one locally symmetric spaces and
  measure classification.
\newblock {\em Israel J. Math.}, 210(1):467--507, 2015.

\bibitem[Yam04]{Yaman2004topological}
Asli Yaman.
\newblock A topological characterisation of relatively hyperbolic groups.
\newblock {\em J. Reine Angew. Math.}, 566:41--89, 2004.

\bibitem[Yan19]{yang2019statistically}
Wen-yuan Yang.
\newblock Statistically convex-cocompact actions of groups with contracting
  elements.
\newblock {\em Int. Math. Res. Not. IMRN}, (23):7259--7323, 2019.

\bibitem[Yan24]{yang2024conformal}
Wen-yuan Yang.
\newblock Conformal dynamics at infinity for groups with contracting elements.
\newblock {\em arXiv preprint arXiv:2208.04861}, 2024.

\end{thebibliography}

\end{document}